\documentclass[12pt,a4paper,leqno]{amsart}
\usepackage[utf8]{inputenc}
\pdfoutput=1

\usepackage{amssymb,amsfonts,mathtools,mathrsfs}   
\usepackage{graphicx}
\usepackage{placeins}
\usepackage{enumerate}
\usepackage{tikz}
\usepackage{setspace}
\usepackage{xcolor}
\usepackage{hyperref}
\usepackage{theoremref}

\usepackage[
  margin=1.5cm,
  includefoot,
  footskip=30pt,
]{geometry}

\usepackage[
backend=biber,
style=alphabetic,
maxbibnames=99, 
]{biblatex}
\newtheorem{thm}{Theorem}[section]
\newtheorem{cor}[thm]{Corollary} 
\newtheorem{lem}[thm]{Lemma} 
\newtheorem{prop}[thm]{Proposition} 
 
\newtheorem*{claim}{Claim}

\theoremstyle{definition} 
\newtheorem{defn}[thm]{Definition}

\theoremstyle{remark}
\newtheorem{remarkx}[thm]{Remark}
\newenvironment{remark}
  {\pushQED{\qed}\remarkx}
  {\popQED\endremarkx}

\numberwithin{equation}{section}

\usepackage{comment}
\excludecomment{notetoself}

\newcommand{\R}{\mathbb{R}}

\newcommand{\dist}{\operatorname{dist}}

\newcommand{\PV}{\mathrm{P.V.}}

\newcommand{\dd}{\mathop{}\!d} 

\def\XXint#1#2#3{{\setbox0=\hbox{$#1{#2#3}{\int}$}
\vcenter{\hbox{$#2#3$}}\kern-.5\wd0}}

\newenvironment{PDE}
	{ \left \{
	\begin{array}{r@{ \ }l @{\quad \: \;} l}
	}
	{
	\end{array} \right . 
	}

\newcommand{\Anorm}[1]{\| #1 \|_{\mathscr{A}_s(\R^n)}}

\renewcommand{\epsilon}{\varepsilon}

\begin{document}

\title[On the Harnack inequality for antisymmetric \(s\)-harmonic functions]{On the Harnack inequality \\
for antisymmetric \(s\)-harmonic functions}

\author[S. Dipierro]{Serena Dipierro}
\address{Serena Dipierro: Department of Mathematics and Statistics,
The University of Western Australia, 35 Stirling Highway, Crawley, Perth, WA 6009, Australia}
\email{serena.dipierro@uwa.edu.au}

\author[J. Thompson]{Jack Thompson}
\address{Jack Thompson: Department of Mathematics and Statistics,
The University of Western Australia, 35 Stirling Highway, Crawley, Perth, WA 6009, Australia}
\email{jack.thompson@research.uwa.edu.au}

\author[E. Valdinoci]{Enrico Valdinoci}
\address{Enrico Valdinoci: Department of Mathematics and Statistics, The University of Western Australia, 35 Stirling Highway, Crawley, Perth, WA 6009, Australia}
\email{enrico.valdinoci@uwa.edu.au}


\date{\today}

\dedicatory{}

\maketitle

\begin{abstract}
We prove the Harnack inequality for antisymmetric \(s\)-harmonic functions, and more generally for
solutions of fractional equations with zero-th order terms, in a general domain. This may be used in conjunction with the method of moving planes to obtain quantitative stability results for symmetry and overdetermined problems for semilinear equations driven by the fractional Laplacian.

The proof is split into two parts: an interior Harnack inequality away from the plane of symmetry, and a boundary Harnack inequality close to the plane of symmetry. We prove these results by first
establishing the weak Harnack inequality for super-solutions and local boundedness for sub-solutions in both the interior and boundary case.

En passant, we also obtain a new mean value formula for antisymmetric $s$-harmonic functions.
\end{abstract}

\section{Introduction and main results}
Since the groundbreaking work \cite{harnack1887basics}, Harnack-type inequalities have become an essential tool in the analysis of partial differential equations (\textsc{PDE}). In \cite{MR1729395,MR3522349,MR3481178}, Harnack inequalities have been applied to antisymmetric functions---functions that are odd with respect to reflections across a given plane, for more detail see Section \ref{LenbLVkk}---in conjunction with the method of moving planes to obtain stability results for local overdetermined problems including \emph{Serrin's overdetermined problem} and the \emph{parallel surface problem}; see also \cite{MR3932952}.

In recent years, there has been consolidated interest in overdetermined problems driven by nonlocal operators, particularly the fractional Laplacian; however, standard Harnack inequalities for such operators are incompatible with antisymmetric functions since they require functions to be non-negative in all of \(\R^n\). In this paper, we will address this problem by proving a Harnack inequality for antisymmetric \(s\)-harmonic functions with zero-th order terms which only requires non-negativity in a halfspace. By allowing zero-th order terms this result is directly applicable to symmetry and overdetermined problems for semilinear equations driven by the fractional Laplacian.

\subsection{Background} Fundamentally the original Harnack inequality for elliptic \textsc{PDE} is a quantitative formulation of the strong maximum principle and directly gives, among other things, Liouville's theorem, the removable singularity theorem, compactness results, and Hölder estimates for weak and viscosity solutions, see~\cite{MR1814364,MR2597943}. This has led it to be extended to a wide variety of other settings. For local \textsc{PDE} these include:
linear parabolic \textsc{PDE}~\cite{MR2597943,MR1465184}; quasilinear and fully nonlinear
elliptic \textsc{PDE}~\cite{MR170096,MR226198,MR1351007}; quasilinear and fully nonlinear
parabolic \textsc{PDE}~\cite{MR244638,MR226168}; and in connection with curvature and geometric
flows such as Ricci flow on Riemannian manifolds~\cite{MR431040,MR834612,MR2251315}.
An extensive survey on the Harnack inequality for local \textsc{PDE} is~\cite{MR2291922}.

For equations arising from jump processes, colloquially known as nonlocal \textsc{PDE}, the first Harnack inequality is due to Bogdan~\cite{MR1438304} who proved the boundary Harnack inequality for the fractional Laplacian. The fractional Laplacian is the prototypical example of a nonlocal operator and is defined by \begin{align}
(-\Delta)^s u(x) &= c_{n,s} \PV \int_{\R^n} \frac{u(x)-u(y)}{\vert x - y \vert^{n+2s}} \dd y  \label{ZdAlT}
\end{align}
where~$s\in(0,1)$, \(c_{n,s}\) is a positive normalisation constant (see~\cite{MR3916700} for more details, particularly Proposition~5.6 there) and \(\PV\) refers to the Cauchy principle value. The result in~\cite{MR1438304} was only valid for Lipschitz domains, but this was later extended to any bounded domain in~\cite{MR1719233}.

Over the proceeding decade there were several papers proving Harnack inequalities for more general jump processes including~\cite{MR1918242,MR3271268}, see also~\cite{MR2365478}. For fully nonlinear nonlocal \textsc{PDE}, a Harnack inequality was established in~\cite{MR2494809}, see also~\cite{MR2831115}. More recently, in~\cite{MR4023466} the boundary Harnack inequality was proved for nonlocal \textsc{PDE} in non-divergence form.

As far as we are aware, the only nonlocal Harnack inequality for antisymmetric functions in the literature is in~\cite{ciraolo2021symmetry} where a boundary Harnack inequality was established for antisymmetric \(s\)-harmonic functions in a small ball centred at the origin. Our results generalise this to arbitrary symmetric domains and to equations with zero-th order terms. 

\subsection{Main results} \label{K4mMv} Let us now describe in detail our main results. First, we will introduce some useful notation. Given a point \(x=(x_1,\dots , x_n) \in \R^n\), we will write \(x=(x_1,x')\) with \(x'=(x_2,\dots , x_n)\) and we denote by \(x_\ast\) the reflection of \(x\) across the plane \(\{x_1=0\}\), that is, \(x_\ast = (-x_1,x')\). Then we call a function \(u:\R^n \to \R\) \emph{antisymmetric} if \begin{align*}
u(x_\ast) = -u(x) \qquad \text{for all } x\in \R^n. 
\end{align*} We will also denote by \(\R^n_+\) the halfspace \(\{x\in \R^n \text{ s.t. } x_1>0\}\) and, given \(A\subset \R^n\), we let \(A^+ := A \cap \R^n_+\). Moreover, we will frequently make use of the
functional space \(\mathscr A_s(\R^n)\) which we define to be the set of all antisymmetric functions \(u\in L^1_{\mathrm {loc}}(\R^n)\) such that \begin{align}
\Anorm{u}:= \int_{\R^n_+} \frac{x_1 \vert u(x) \vert}{1+ \vert x \vert^{n+2s+2}}\dd x \label{QhKBn}
\end{align} is finite. The role that the new space \(\mathscr A_s(\R^n)\) plays will be explained in more detail later in this section.

Our main result establishes the Harnack inequality\footnote{The arguments presented here are quite general and can be suitably extended to other integro-differential operators. For the specific case of the fractional Laplacian, the extension problem can provide alternative ways to prove some of the results presented here. We plan to treat this case extensively in a forthcoming paper, but in Appendix~\ref{APPEEXT:1} here
we present a proof of~\eqref{LA:PAKSM} specific for the fractional Laplacian and~$c:=0$
that relies on extension methods.}
in a general symmetric domain \(\Omega\).

\begin{thm} \thlabel{Hvmln}
Let \(\Omega \subset \R^n\) and \(\tilde \Omega \Subset \Omega\) be bounded domains that are symmetric with respect to \(\{x_1=0\}\), and let \(c\in L^\infty(\Omega^+)\). Suppose that \(u \in C^{2s+\alpha}(\Omega)\cap \mathscr{A}_s(\R^n) \) for some \(\alpha >0\) with \(2s+\alpha\) not an integer, \(u\) is non-negative in \(\R^n_+\), and satisfies \begin{align*}
(-\Delta)^s u +cu&=0 \qquad \text{in }\Omega^+.
\end{align*}

Then there exists \(C>0\) depending only on \(\Omega\), \(\tilde \Omega\), \(\| c\|_{L^\infty(\Omega^+)}\), \(n\), and \(s\) such that \begin{equation}\label{LA:PAKSM}
\sup_{\tilde \Omega^+} \frac{u(x)}{x_1}  \leqslant C \inf_{ \tilde \Omega^+} \frac{u(x)}{x_1} .
\end{equation}  Moreover, \(\inf_{\tilde \Omega^+} \frac{u(x)}{x_1}\) and \(\sup_{\tilde \Omega^+} \frac{u(x)}{x_1}\) are comparable to \(\Anorm{u}\).
\end{thm}

Here and throughout this document we use the convention that if \(\beta>0\) and \(\beta\) is not an integer then \(C^\beta (\Omega)\) denotes the Hölder space \(C^{k,\beta'}(\Omega)\) where \(k\) is the integer part of \(\beta\) and \(\beta'=\beta-k \in (0,1)\).

One can compare the Harnack inequality in Theorem~\ref{Hvmln} here
with previous results in the literature, such as Proposition~6.1 in~\cite{MR4308250},
which can be seen as a boundary Harnack inequality for antisymmetric functions,
in a rather general fractional elliptic setting. We point out that Theorem~\ref{Hvmln} here
holds true without any sign assumption on~$c$ (differently from Proposition~6.1 in~\cite{MR4308250}
in which a local sign assumption\footnote{The sign assumption on \(c\) in Proposition~6.1 of~\cite{MR4308250} is required since the result was for unbounded domains. On the other hand, our result is for bounded domains which is why this assumption is no longer necessary. } on~$c$ was taken), for all~$s\in(0,1)$ (while the analysis of Proposition~6.1 in~\cite{MR4308250}
was focused on the case~$s\in\left[\frac12,1\right)$), and in every dimension~$n$ (while Proposition~6.1 in~\cite{MR4308250} dealt with the case~$n=1$).

We would like to emphasise that, in contrast to the standard Harnack inequality for the fractional Laplacian, \thref{Hvmln} only assumes \(u\) to be non-negative in the halfspace \(\R^n_+\) which is a more natural assumption for antisymmetric functions. Also, note that the assumption that \(\tilde \Omega \Subset \Omega\) allows~\(\tilde \Omega^+\) to touch \(\{x_1=0\}\), but prevents \(\tilde \Omega^+\) from touching the portion of \(\partial\Omega\) that is not on \(\{x_1=0\}\); it is in this sense that we will sometimes refer to \thref{Hvmln} (and later \thref{C35ZH}) as a \emph{boundary} Harnack inequality for antisymmetric functions.

Interestingly, the quantity \(\Anorm{u}\) arises very naturally in the context of antisymmetric functions and plays the same role that \begin{align}
\| u\|_{\mathscr L_s (\R^n)} := \int_{\R^n} \frac{\vert u (x) \vert }{1+\vert x \vert^{n+2s}} \dd x \label{ACNAzPn8}
\end{align} plays in the non-antisymmetric nonlocal Harnack inequality, see for example \cite{MR4023466}. To our knowledge, \(\Anorm{\cdot}\) is new in the literature.

A technical aspect of \thref{Hvmln}, however, is that the fractional Laplacian is in general not defined if we simply have \(u\) is \(C^{2s+\alpha}\) and \(\Anorm{u} < +\infty\). This leads us to introduce the following new definition of the fractional Laplacian for functions in \(\mathscr A_s(\R^n)\) which we will use throughout the remainder of this paper. 

\begin{defn} \thlabel{mkG4iRYH}
Let \(x\in \R^n_+\) and suppose that \(u \in \mathscr A_s(\R^n)\) is \( C^{2s+\alpha}\) in a neighbourhood of \(x\) for some \(\alpha>0\) with \(2s+\alpha\) not an integer. The fractional Laplacian of \(u\) at \(x\), which we denote as usual by \((-\Delta)^su(x)\), is defined by \begin{align}
(-\Delta)^su(x) = c_{n,s} \lim_{\varepsilon\to 0^+} \int_{\R^n_+ \setminus B_\varepsilon(x)} \left(
\frac 1 {\vert x - y \vert^{n+2s}} - \frac 1 {\vert x_\ast- y \vert^{n+2s}}\right) \big(u(x)-u(y)\big)\dd y  + \frac {c_{1,s}} s  u(x)x_1^{-2s} \label{OopaQ0tu}
\end{align} where \(c_{n,s}\) is the constant from~\eqref{ZdAlT}.
\end{defn}

We will motivate \thref{mkG4iRYH} in Section \ref{LenbLVkk} as well as verify it is well-defined. We also prove in Section \ref{LenbLVkk} that if \(\| u\|_{\mathscr L_s(\R^n)}\) is finite, $u$ is \( C^{2s+\alpha}\) in a neighbourhood of \(x\), and antisymmetric then \thref{mkG4iRYH} agrees with~\eqref{ZdAlT}. 

See also~\cite{MR3453602, MR4108219, MR4030266, MR4308250} where maximum principles in the setting
of antisymmetric solutions of integro-differential equations have been established by exploiting
the antisymmetry of the functions as in~\eqref{OopaQ0tu}.

It is also worth mentioning that the requirement that \(u\) is antisymmetric in \thref{Hvmln} cannot be entirely removed. In Appendix \ref{j4njb}, we construct a sequence of functions that explicitly demonstrates this. Moreover, \thref{Hvmln} is also false if one only assumes \(u\geqslant0\) in \(\Omega^+\) as proven in \cite[Corollary 1.3]{RoleAntisym2022}.

To obtain the full statement of \thref{Hvmln} we divide the proof into two parts: an interior Harnack inequality and a boundary Harnack inequality close to \(\{x_1=0\}\). The interior Harnack inequality is given as follows.

\begin{thm} \thlabel{DYcYH}
Let \(\rho\in(0,1)\) and \(c\in L^\infty(B_\rho(e_1))\). Suppose that \(u \in C^{2s+\alpha}(B_\rho(e_1))\cap \mathscr{A}_s(\R^n) \) for some \(\alpha >0\) with \(2s+\alpha\) not an integer, \(u\) is non-negative in \(\R^n_+\), and satisfies \begin{align*}
(-\Delta)^s u +cu =0 \qquad\text{in } B_{\rho}(e_1).
\end{align*}

Then there exists \(C_\rho>0\) depending only on \(n\), \(s\), \(\| c \|_{L^\infty(B_\rho(e_1))}\), and \(\rho\) such that \begin{align*}
\sup_{B_{\rho/2}(e_1)} u \leqslant C_\rho \inf_{B_{\rho/2}(e_1)} u.
\end{align*} Moreover, both the quantities \(\sup_{B_{\rho/2}(e_1)} u \) and \(\inf_{B_{\rho/2}(e_1)} u \) are comparable to \(\Anorm{u}\).
\end{thm}

For all \(r>0\), we write \(B_r^+ := B_r \cap \R^n_+\). Then the following result is the antisymmetric boundary Harnack inequality.

\begin{thm} \thlabel{C35ZH}
Let \(\rho>0\) and \(c\in L^\infty(B_\rho^+)\). Suppose that \(u \in C^{2s+\alpha}(B_\rho)\cap \mathscr{A}_s(\R^n) \) for some \(\alpha >0\) with \(2s+\alpha\) not an integer, \(u\) is non-negative in \(\R^n_+\), and satisfies \begin{align*}
(-\Delta)^s u +cu&=0 \qquad \text{in } B_\rho^+.
\end{align*}

Then there exists \(C_\rho>0\) depending only on \(n\), \(s\), \(\| c \|_{L^\infty(B_\rho^+)}\), and \(\rho\) such that\begin{align*}
\sup_{x\in B_{\rho/2}^+} \frac{u(x)}{x_1} \leqslant C_\rho \inf_{x\in B_{\rho/2}^+} \frac{u(x)}{x_1}.
\end{align*}  Moreover, both the quantities \(\sup_{x\in B_{\rho/2}^+} \frac{u(x)}{x_1} \) and \(\inf_{x\in B_{\rho/2}^+} \frac{u(x)}{x_1} \) are comparable to \(\Anorm{u}\).
\end{thm}

To our knowledge, Theorems~\ref{DYcYH} and~\ref{C35ZH} are new in the literature. In the particular case \(c\equiv0\), \thref{C35ZH} was proven in \cite{ciraolo2021symmetry}, but the proof relied on the Poisson representation formula for the Dirichlet problem in a ball which is otherwise unavailable in more general settings.

The proofs of Theorems~\ref{DYcYH} and~\ref{C35ZH} are each split into two halves: for \thref{DYcYH} we prove an interior weak Harnack inequality for super-solutions (\thref{MT9uf}) and interior local boundedness of sub-solutions (\thref{guDQ7}). Analogously, for \thref{C35ZH} we prove a boundary weak Harnack inequality for super-solutions (\thref{SwDzJu9i}) and boundary local boundedness of sub-solutions (\thref{EP5Elxbz}). The proofs of Propositions~\ref{MT9uf}, \ref{guDQ7}, \ref{SwDzJu9i}, and~\ref{EP5Elxbz} make use of general barrier methods which take inspiration from \cite{MR4023466,MR2494809,MR2831115}; however, these methods require adjustments which take into account the antisymmetry assumption. 

Once Theorems~\ref{DYcYH} and \ref{C35ZH} have been established, \thref{Hvmln} follows easily from a standard covering argument. For completeness, we have included this in Section \ref{lXIUl}. 

Finally, in Appendix~\ref{4CEly} we provide an alternate elementary proof of \thref{C35ZH} in the particular case \(c\equiv0\). As in \cite{ciraolo2021symmetry}, our proof relies critically on the Poisson representation formula in a ball for the fractional Laplacian, but the overall strategy of our proof is entirely different. This was necessary to show that \(\sup_{x\in B_{\rho/2}^+} \frac{u(x)}{x_1} \) and \(\inf_{x\in B_{\rho/2}^+} \frac{u(x)}{x_1} \) are comparable to \(\Anorm{u}\) which was not proven in \cite{ciraolo2021symmetry}. Our proof makes use of a new mean-value formula for antisymmetric \(s\)-harmonic functions, which we believe to be interesting
in and of itself. 

The usual mean value formula for \(s\)-harmonic functions says that if \(u\) is \(s\)-harmonic in \(B_1\) then we have that \begin{align}
u(0) =  \gamma_{n,s} \int_{\R^n\setminus B_r} \frac{r^{2s} u(y)}{(\vert y \vert^2-r^2)^s\vert y \vert^n}\dd y \qquad \text{for all }r\in(0, 1] \label{K8kw5rpi}
\end{align}where
\begin{equation}\label{sry6yagamma098765} \gamma_{n,s}:= \frac{\sin(\pi s)\Gamma (n/2)}{\pi^{\frac n 2 +1 }}.
\end{equation}

{F}rom antisymmetry, however, both the left-hand side and the right-hand side of~\eqref{K8kw5rpi} are zero irrespective of the fact \(u\) is \(s\)-harmonic. It is precisely this observation that leads to the consideration of \(\partial_1u(0)\) instead of \(u(0)\) which is more appropriate when \(u\) is antisymmetric.

\begin{prop} \thlabel{cqGgE}
Let \(u\in C^{2s+\alpha}(B_1) \cap \mathscr L_s(\R^n)\) with \(\alpha>0\) and \(2s+\alpha\) not an integer. Suppose that~\(u\) is antisymmetric and \(r\in(0, 1]\).  If \((-\Delta)^s u = 0 \) in \(B_1\) then  \begin{align*}
\frac{\partial u}{\partial x_1} (0) &=  2n \gamma_{n,s} \int_{\R^n_+ \setminus B_r^+} \frac{r^{2s}y_1 u(y)}{(\vert y \vert^2 - r^2)^s\vert y \vert^{n+2}} \dd y ,
\end{align*}
with~$\gamma_{n,s}$ given by~\eqref{sry6yagamma098765}.
\end{prop}

\subsection{Organisation of paper}

The paper is organised as follows. In Section \ref{LenbLVkk}, we motivate \thref{mkG4iRYH} and establish some technical lemmata regarding this definition. In Section~\ref{q7hKF}, we
prove \thref{DYcYH} and, in Section~\ref{TZei6Wd4}, we prove \thref{C35ZH}. In Section~\ref{lXIUl}, we prove the main result \thref{Hvmln}. In Appendix~\ref{j4njb}, we construct a sequence that demonstrates the antisymmetric assumption in \thref{Hvmln} cannot be removed. Finally, in Appendix~\ref{4CEly}, we prove
Proposition~\ref{cqGgE}, and with this, we provide an elementary alternate proof of \thref{C35ZH} for the particular case~\(c\equiv0\).

\section{The antisymmetric fractional Laplacian} \label{LenbLVkk}

Let \(n\) be a positive integer and \(s\in (0,1)\). The purpose of this section is to motivate \thref{mkG4iRYH} as well as prove some technical aspects of the definition such as that it is well-defined and that it coincides with~\eqref{ZdAlT} when \(u\) is sufficiently regular. Recall that, given a point \(x=(x_1,\dots , x_n) \in \R^n\), we will write \(x=(x_1,x')\) with \(x'=(x_2,\dots , x_n)\) and we denote by \(x_\ast\) the reflection of \(x\) across the plane \(\{x_1=0\}\), that is, \(x_\ast = (-x_1,x')\). Then we call a function \(u:\R^n \to \R\) \emph{antisymmetric} if \begin{align*}
u(x_\ast) = -u(x) \qquad \text{for all } x\in \R^n. 
\end{align*} It is common in the literature, particularly when dealing with the method of moving planes, to define antisymmetry with respect to a given hyperplane \(T\); however, for our purposes, it is sufficient to take~\(T=\{x_1=0\}\).

For the fractional Laplacian of \(u\), given by~\eqref{ZdAlT}, to be well-defined in a pointwise sense, we need two ingredients: \(u\) needs enough regularity in a neighbourhood of \(x\) to overcome the singularity at \(x\) in the kernel \(\vert x-y\vert^{-n-2s}\), and \(u\) also needs an integrability condition at infinity to account for the integral in~\eqref{ZdAlT} being over \(\R^n\). For example, if \(u\) is \(C^{2s+\alpha}\) in a neighbourhood of \(x\) for some~\(\alpha >0\) with \(2s+\alpha\) not an integer and \(u\in\mathscr L_s(\R^n)\) where \(\mathscr L_s (\R^n)\) is the set of functions \(v\in L^1_{\mathrm{loc}}(\R^n)\) such that \(\| v\|_{\mathscr L_s(\R^n)}  <+\infty\) (recall \(\| \cdot \|_{\mathscr L_s(\R^n)}\) is given by~\eqref{ACNAzPn8}) then \((-\Delta)^su\) is well-defined at \(x\) and is in fact continuous there, \cite[Proposition 2.4]{MR2270163}. 

In the following proposition, we show that if \(u\in \mathscr L_s(\R^n)\) is antisymmetric and \(u\) is \(C^{2s+\alpha}\) in a neighbourhood of \(x\) for some \(\alpha >0\) with \(2s+\alpha\) not an integer then \(u\) satisfies~\eqref{OopaQ0tu}. This simultaneously motivates \thref{mkG4iRYH} and demonstrates this definition does generalise the definition of the fractional Laplacian when the given function is antisymmetric.

\begin{prop} \thlabel{eeBLRjcZ}
Let \(x\in \R^n_+\) and \(u \in \mathscr L_s(\R^n)\) be an antisymmetric function that is \(C^{2s+\alpha}\) in a neighbourhood of \(x\) for some \(\alpha>0\) with \(2s+\alpha\) not an integer. Then~\eqref{ZdAlT} and \thref{mkG4iRYH} coincide. 
\end{prop}

\begin{proof}
Let~$x=(x_1,x')\in \R^n_+$ and take~$\delta\in\left(0,\frac{x_1}2\right)$ such that~$u$
is~$C^{2s+\alpha}$ in~$B_\delta(x)\subset\R^n_+$. Furthermore, let \begin{align}
(-\Delta)^s_\delta u(x):= c_{n,s}\int_{\R^n  \setminus B_\delta(x)} \frac{u(x)-u(y)}{\vert x - y\vert^{n+2s}} \dd y .\label{WUuXb0Hk}
\end{align}
By the regularity assumptions on \(u\), the integral in~\eqref{ZdAlT} is well-defined and \begin{align*}
 \lim_{\delta\to0^+}(-\Delta)^s_\delta u(x) = c_{n,s} \PV \int_{\R^n}& \frac{u(x)-u(y)}{\vert x - y \vert^{n+2s}} \dd y =(-\Delta)^s u(x). 
\end{align*}

Splitting the integral in~\eqref{WUuXb0Hk} into two integrals over \(\R^n_+  \setminus B_\delta(x)\) and \(\R^n_-\) respectively and
then using that \(u\) is antisymmetric in the integral over \(\R^n_-\), we obtain
\begin{align}
(-\Delta)^s_\delta u(x) &= c_{n,s}\int_{\R^n_+ \setminus B_\delta(x)} \frac{u(x)-u(y)}{\vert x - y\vert^{n+2s}} \dd y +c_{n,s} \int_{\R^n_+} \frac{u(x)+u(y)}{\vert x_\ast - y\vert^{n+2s}} \dd y \nonumber \\
&= c_{n,s}\int_{\R^n_+ \setminus B_\delta(x)} \left(
\frac 1 {\vert x - y \vert^{n+2s}} - \frac 1 {\vert x_\ast- y \vert^{n+2s}} \right)\big(u(x)-u(y)\big) \dd y \nonumber \\
&\qquad +2c_{n,s}u(x) \int_{\R^n_+ \setminus B_\delta(x)} \frac{\dd y }{\vert x_\ast - y \vert^{n+2s}} + c_{n,s} \int_{B_\delta(x)} \frac{u(x)+u(y) }{\vert x_\ast - y \vert^{n+2s}}\dd y . \label{cGXQjDIV}
\end{align}

We point out that if~$y\in\R^n_+$ then~$|x_\ast-y|\ge  x_1$,
and therefore
\begin{eqnarray*}
&& \left \vert \int_{B_\delta(x)} \frac{u(x)+u(y) }{\vert x_\ast - y \vert^{n+2s}}\dd y \right \vert\le
\int_{B_\delta(x)} \frac{2|u(x)|+|u(y)-u(x)| }{\vert x_\ast - y \vert^{n+2s}}\dd y\\
&&\qquad\qquad \le \int_{B_\delta(x)} \frac{2|u(x)|+C|x-y|^{\min\{1,2s+\alpha\}} }{\vert x_\ast - y \vert^{n+2s}}\dd y\le  C\big (  \vert u(x) \vert + \delta^{\min\{1,2s+\alpha\}} \big ) \delta^n,
\end{eqnarray*}
for some~$C>0$ depending on~$n$, $s$, $x$, and the norm of~$u$ in a neighbourood of~$x$.

As a consequence, sending \(\delta \to 0^+\) in~\eqref{cGXQjDIV}
and using the Dominated Convergence Theorem,
we obtain \begin{align*}
(-\Delta)^s u(x) &= c_{n,s}\lim_{\delta \to 0^+}\int_{\R^n_+ \setminus B_\delta(x)} \left( \frac 1 {\vert x - y \vert^{n+2s}} - \frac 1 {\vert x_\ast- y \vert^{n+2s}} \right)\big(u(x)-u(y)\big) \dd y  \\
&\qquad +2c_{n,s}u(x) \int_{\R^n_+} \frac{\dd y }{\vert x_\ast - y \vert^{n+2s}}. 
\end{align*}

Also, making the change of variables \(z=(y_1/x_1 , (y'-x')/x_1)\) if \(n>1\) and \(z_1=y_1/x_1\) if \(n=1\), we see that \begin{align*}
 \int_{\R^n_+} \frac{\dd y }{\vert x_\ast - y \vert^{n+2s}}  &= x_1^{-2s}  \int_{\R^n_+}   \frac{\dd z}{\vert e_1+ z \vert^{n+2s}}. 
\end{align*}
Moreover, via a direct computation  \begin{equation}\label{sd0w3dewt95b76 498543qdetr57uy54u}
 c_{n,s}   \int_{\R^n_+}   \frac{\dd z}{\vert e_1+ z \vert^{n+2s}} = \frac {c_{1,s}}{2s} 
\end{equation} see \thref{tuMIrhco} below for details. 

Putting together these considerations, we obtain the desired result.
\end{proof}

\begin{remark} \thlabel{tuMIrhco}
Let \begin{align}
\tilde c_{n,s} := c_{n,s}   \int_{\R^n_+}   \frac{\dd z}{\vert e_1+ z \vert^{n+2s}}. \label{CA7GPGvx}
\end{align}The value of \(\tilde c_{n,s}\) is of no consequence to \thref{mkG4iRYH} and so, for this paper, is irrelevant; however, it is interesting to note that \(\tilde c_{n,s}\) can be explicitly evaluated and is, in fact, equal to \((2s)^{-1}c_{1,s}\) (in particular \(\tilde c_{n,s}\) is independent of \(n\) which is not obvious from~\eqref{CA7GPGvx}). Indeed, if \(n>1\) then \begin{align*}
 \int_{\R^n_+}   \frac{\dd z}{\vert e_1+ z \vert^{n+2s}} &= \int_0^\infty \int_{\R^{n-1}}   \frac{\dd z_1\dd z'}{\big ( (z_1+1)^2 + \vert z' \vert^2\big )^{\frac{n+2s}2}}
\end{align*} so, making the change of variables \(\tilde z'= (z_1+1)^{-1} z'\) in the inner integral, we have that \begin{align*}
 \int_{\R^n_+}   \frac{\dd z}{\vert e_1+ z \vert^{n+2s}} &=  \int_0^\infty   \frac 1 {(z_1+1)^{1+2s}} \left(  \int_{\R^{n-1}}   \frac{\dd \tilde z'}{\big ( 1 + \vert \tilde z' \vert^2\big )^{\frac{n+2s}2}} \right) \dd z_1.
\end{align*} By \cite[Proposition 4.1]{MR3916700}, \begin{align*}
\int_{\R^{n-1}}   \frac{\dd \tilde z'}{\big ( 1 + \vert \tilde z' \vert^2\big )^{\frac{n+2s}2}} &= \frac{\displaystyle
\pi^{\frac{n-1}2} \Gamma \big ( \frac{1+2s}2 \big )}{\displaystyle \Gamma \big ( \frac{n+2s}2 \big )}.
\end{align*} Hence, \begin{align*}
\int_{\R^n_+}   \frac{\dd z}{\vert e_1+ z \vert^{n+2s}} &= \frac{\pi^{\frac{n-1}2} \Gamma \big ( \frac{1+2s}2 \big )}{\Gamma \big ( \frac{n+2s}2 \big )} \int_0^\infty   \frac {\dd z_1} {(z_1+1)^{1+2s}} =  \frac{\pi^{\frac{n-1}2} \Gamma \big ( \frac{1+2s}2 \big )}{2s \Gamma \big ( \frac{n+2s}2 \big )}.
\end{align*} This formula remains valid if \(n=1\). Since \(c_{n,s} = s \pi^{-n/2}4^s \Gamma \big ( \frac{n+2s}2 \big )/ \Gamma (1-s)\), see \cite[Proposition~5.6]{MR3916700}, it follows that \begin{align*}
\tilde c_{n,s} &= \frac{2^{2s-1} \Gamma \big ( \frac{1+2s}2 \big )}{\pi^{1/2} \Gamma(1-s) } = \frac { c_{1,s}}{2s}.
\end{align*}
In particular, this proves formula~\eqref{sd0w3dewt95b76 498543qdetr57uy54u},
as desired.
\end{remark}

The key observation from \thref{eeBLRjcZ} is that the kernel \( {\vert x - y \vert^{-n-2s}} -  {\vert x_\ast- y \vert^{-n-2s}}\) decays quicker at infinity than the kernel \(\vert x-y \vert^{-n-2s}\) which is what allows us to relax the assumption~\(u\in \mathscr L_s (\R^n)\). Indeed, if~$x$, $y\in \R^n_+ $ then \(\vert x_\ast - y \vert \geqslant \vert x-y \vert\) and \begin{align}
 \frac 1 {\vert x - y \vert^{n+2s}} - \frac 1 {\vert x_\ast- y \vert^{n+2s}} &= \frac{n+2s} 2 \int_{\vert x - y \vert^2}^{\vert x_\ast -y \vert^2} t^{-\frac{n+2s+2}2} \dd t \nonumber \\
 &\leqslant \frac{n+2s} 2  \big ( \vert x_\ast -y \vert^2-\vert x -y \vert^2 \big ) \frac 1 {\vert x - y \vert^{n+2s+2}} \nonumber \\
 &= 2(n+2s) \frac{x_1y_1}{\vert x - y \vert^{n+2s+2}}. \label{LxZU6}
\end{align} Similarly, for all~$x$, $y\in \R^n_+$,\begin{align}
\frac 1 {\vert x - y \vert^{n+2s}} - \frac 1 {\vert x_\ast- y \vert^{n+2s}} &\geqslant 2(n+2s)  \frac{x_1y_1}{\vert x_\ast - y \vert^{n+2s+2}}. \label{buKHzlE6}
\end{align} Hence, if \(\vert x\vert <C_0\) and if \(|x-y|>C_1\) for some \(C_0,C_1>0\) independent of \(x,y\), we find that \begin{align}
\frac{C^{-1} x_1y_1}{1+\vert y \vert^{n+2s+2}} \leqslant \frac 1 {\vert x - y \vert^{n+2s}} - \frac 1 {\vert x_\ast- y \vert^{n+2s}} \leqslant \frac{Cx_1y_1}{1+\vert y \vert^{n+2s+2}} \label{Zpwlcwex}
\end{align} which motivates the consideration of \(\mathscr A_s(\R^n)\).

Our final lemma in this section proves that \thref{mkG4iRYH} is well-defined, that is, the assumptions of \thref{mkG4iRYH} are enough to guarantee \eqref{OopaQ0tu} converges.

\begin{lem} \thlabel{zNNKjlJJ}
Let \(x \in \R^n_+\) and \(r\in (0,x_1/2)\) be sufficiently small so that \(B:=B_r(x) \Subset\R^n_+\). If \(u \in C^{2s+\alpha}(B) \cap  \mathscr A_s(\R^n)\) for some \(\alpha>0\) with \(2s+\alpha\) not an integer then \((-\Delta)^su(x)\) given by \eqref{OopaQ0tu} is well-defined and there exists \(C>0\) depending on \(n\), \(s\), \(\alpha\), \(r\), and \(x_1\) such that \begin{align*}
\big \vert (-\Delta)^s u(x) \big  \vert &\leqslant C \big (  \| u\|_{C^{\beta}(B)} + \Anorm{u} \big ) 
\end{align*} where \(\beta = \min \{ 2s+\alpha ,2\}\).
\end{lem}

We remark that the constant \(C\) in \thref{zNNKjlJJ} blows up as \(x_1\to 0^+\). 

\begin{proof}[Proof of Lemma~\ref{zNNKjlJJ}]
Write \( (-\Delta)^s u(x) = I_1 +I_2 + s^{-1} c_{1,s} u(x) x_1^{-2s}\) where  \begin{align*}
I_1 &:=  c_{n,s}\lim_{\varepsilon \to 0^+} \int_{B\setminus B_\varepsilon(x)} \left(
\frac 1 {\vert x - y \vert^{n+2s}} - \frac 1 {\vert x_\ast- y \vert^{n+2s}}\right) \big(u(x)-u(y)\big)\dd y\\
\text{ and } \qquad I_2 &:=  c_{n,s}\int_{\R^n_+\setminus B} \left(\frac 1 {\vert x - y \vert^{n+2s}} - \frac 1 {\vert x_\ast- y \vert^{n+2s}}\right) \left(u(x)-u(y)\right)\dd y.
\end{align*}

\emph{For \(I_1\):} If \(y\in \R^n_+\) then \({\vert x - y \vert^{-n-2s}} -  {\vert x_\ast- y \vert^{-n-2s}} \) is only singular at \(x\), so we may write \begin{align*}
I_1 &= c_{n,s}\lim_{\varepsilon \to 0^+} \int_{B\setminus B_\varepsilon(x)}\frac {u(x)-u(y)} {\vert x - y \vert^{n+2s}}\dd y - \int_B \frac {u(x)-u(y)} {\vert x_\ast- y \vert^{n+2s}}\dd y.
\end{align*}If \(\chi_B\) denotes the characteristic function of \(B\) then \(\bar  u := u\chi_B\in C^{2s+\alpha}(B)\cap L^\infty(\R^n)\), so \((-\Delta)^s\bar u(x)\) is defined and \begin{align*}
(-\Delta)^s\bar u(x) &= c_{n,s}\PV \int_{\R^n} \frac{u(x) - u(y)\chi_B(y)}{\vert x-y\vert^{n+2s}} \dd y \\
&= c_{n,s} \PV \int_B \frac{u(x) - u(y)}{\vert x-y\vert^{n+2s}} \dd y + c_{n,s} u(x)\int_{\R^n\setminus B} \frac{\dd y }{\vert x-y\vert^{n+2s}}\\
&=  c_{n,s}\PV \int_B \frac{u(x) - u(y)}{\vert x-y\vert^{n+2s}} \dd y +  \frac 1{2s}c_{n,s} n \vert B_1 \vert r^{-2s} u(x) .
\end{align*}Hence, \begin{align*}
I_1&= (-\Delta)^s\bar u(x) -\frac 1{2s}c_{n,s} n \vert B_1 \vert r^{-2s} u(x) -c_{n,s} \int_B \frac {u(x)-u(y)} {\vert x_\ast- y \vert^{n+2s}}\dd y 
\end{align*} which gives that \begin{align*}
\vert I_1 \vert &\leqslant  \big \vert (-\Delta)^s\bar u(x) \big \vert +C \| u\|_{L^\infty(B)} +2\| u\|_{L^\infty(B)} \int_B \frac 1 {\vert x_\ast- y \vert^{n+2s}}\dd y  \\
&\leqslant   \big \vert (-\Delta)^s\bar u(x) \big \vert  +C\| u\|_{L^\infty(B)}.
\end{align*}
Furthermore, if \(\alpha <2(1-s)\) then by a Taylor series expansion  \begin{align*}
\big \vert (-\Delta)^s\bar u (x) \big \vert &\leqslant C \int_{\R^n} \frac{\vert 2\bar u(x)-\bar u(x+y) - \bar u(x-y) \vert }{\vert y \vert^{n+2s}} \dd y \\
&\leqslant C [\bar u]_{C^{2s+\alpha}(B)}  \int_{B_1} \frac{\dd y  }{\vert y \vert^{n-\alpha}}  + C \| \bar u \|_{L^\infty(\R^n)} \int_{\R^n\setminus B_1} \frac{\dd y }{\vert y \vert^{n+2s}} \\
&\leqslant C  \big (  \| \bar u\|_{C^{2s+\alpha}(B)} + \| \bar u \|_{L^\infty(\R^n)} \big ) \\
&=C  \big ( \| u\|_{C^{2s+\alpha}(B)} + \| u \|_{L^\infty(B)} \big )  .
\end{align*} If \(\alpha \geqslant 2(1-s)\) then this computation holds with \([\bar u]_{C^{2s+\alpha}(B)} \) replaced with \(\|u\|_{C^2(B)}\).

\emph{For \(I_2\):} By~\eqref{LxZU6}, \begin{align*}
\vert I_2 \vert &\leqslant Cx_1\int_{\R^n_+\setminus B} \frac {y_1\big (\vert u(x) \vert + \vert u(y) \vert  \big ) } {\vert x - y\vert^{n+2s+2}} \dd y \\
&\leqslant C \| u \|_{L^\infty(B)}\int_{\R^n_+\setminus B} \frac {y_1} {\vert x - y\vert^{n+2s+2}} \dd y + C \Anorm{u} \\
&\leqslant C \big (\| u \|_{L^\infty(B)}+ \Anorm{u} \big ) .
\end{align*}

Combining the estimates for \(I_1\) and \(I_2\) immediately gives the result. 
\end{proof}

\section{Interior Harnack inequality and proof of Theorem~\ref{DYcYH}}\label{q7hKF}

The purpose of this section is to prove \thref{DYcYH}. Its proof is split into two parts: the interior weak Harnack inequality for super-solutions (\thref{MT9uf}) and interior local boundedness for sub-solutions (\thref{guDQ7}).

\subsection{The interior weak Harnack inequality}

The interior weak Harnack inequality for super-solutions is given in \thref{MT9uf} below. 

\begin{prop} \thlabel{MT9uf} 
Let \(M\in \R \), \(\rho \in(0,1)\), and \(c\in L^\infty(B_\rho(e_1))\). Suppose that \(u \in C^{2s+\alpha}(B_\rho(e_1))\cap \mathscr{A}_s(\R^n)\) for some \(\alpha>0\) with \(2s+\alpha\) not an integer, \(u\) is non-negative in \(\R^n_+\), and satisfies  \begin{align*}
(-\Delta)^su +cu &\geqslant -M \qquad \text{in } B_\rho(e_1).
\end{align*}

Then there exists \(C_\rho>0\) depending only on \(n\), \(s\), \(\| c \|_{L^\infty(B_\rho(e_1))}\), and \(\rho\) such that \begin{align*}
\Anorm{u} &\leqslant C_\rho \left( \inf_{B_{\rho/2}(e_1)} u + M \right) . 
\end{align*}
\end{prop}

The proof of \thref{MT9uf} is a simple barrier argument that takes inspiration from \cite{MR4023466}. We will begin by proving the following proposition, \thref{YLj1r},  which may be viewed as a rescaled version of \thref{MT9uf}.  We will require \thref{YLj1r} in the second part of the section when we prove the interior local boundedness of sub-solutions (\thref{guDQ7}). 

\begin{prop} \thlabel{YLj1r} 
Let \(M\in \R \), \(\rho \in(0,1)\), and \(c\in L^\infty(B_\rho(e_1))\). Suppose that \(u \in C^{2s+\alpha}(B_\rho(e_1))\cap \mathscr{A}_s(\R^n)\) for some \(\alpha>0\) with \(2s+\alpha\) not an integer, \(u\) is non-negative in \(\R^n_+\), and satisfies  \begin{align*}
(-\Delta)^su +cu &\geqslant -M \qquad \text{in } B_\rho(e_1).
\end{align*}

Then there exists \(C>0\) depending only on \(n\), \(s\), and~\(\rho^{2s}\| c \|_{L^\infty(B_\rho(e_1))}\),
such that  \begin{align*}
 \frac 1 {\rho^n} \int_{B_{\rho/2}(e_1)} u(x) \dd x &\leqslant C \left( \inf_{B_{\rho/2}(e_1)} u + M \rho^{2s}\right) . 
\end{align*} Moreover, the constant \(C\) is of the form \begin{align}
C = C' \left(1+ \rho^{2s} \| c \|_{L^\infty(B_\rho(e_1))}\right) \label{CyJJQHrF}
\end{align} with \(C'>0\) depending only on \(n\) and \(s\). 
\end{prop}

Before we give the proof of \thref{YLj1r}, we introduce some notation.
Given~$ \rho \in(0,1)$, we define
$$  T_\rho:= \left\{x_1=-\frac1{\rho}\right\}\qquad{\mbox{and}}\qquad
H_\rho:=\left\{x_1>-\frac1{\rho}\right\}.$$
We also let~$Q_\rho$ be the reflection across the hyperplane~$T_\rho$, namely
$$Q_\rho(x):=x-2(x_1+1/\rho)e_1\qquad{\mbox{for all }} x\in \R^n. $$
With this, we establish the following lemma.

\begin{lem} \thlabel{rvZJdN88}
Let \(\rho \in(0,1)\) and \(\zeta\) be a smooth cut-off function such that \begin{align*}
\zeta \equiv 1 \text{ in } B_{1/2}, \quad \zeta \equiv 0 \text{ in } \R^n \setminus B_{3/4}, \quad{\mbox{ and }}
\quad
0\leqslant \zeta \leqslant 1.
\end{align*} Define \(\varphi^{(1)} \in C^\infty_0(\R^n)\) by \begin{align*}
\varphi^{(1)}(x) := \zeta (x) - \zeta \big (  Q_\rho(x) \big ) \qquad \text{for all } x\in \R^n.
\end{align*}

Then \(\varphi^{(1)}\) is antisymmetric with respect to \(T_\rho:= \{x_1=-1/\rho\}\) and there exists \(C>0\) depending only on \(n\) and \(s\) (but not on \(\rho\)) such that \(\| (-\Delta)^s \varphi^{(1)} \|_{L^\infty(B_{3/4})} \leqslant C\).
\end{lem}

\begin{proof}
Since \(Q_\rho(x)\) is the reflection of \(x\in \R^n\) across \(T_\rho\), we immediately obtain that \(\varphi^{(1)} \) is antisymmetric with respect to the plane \(T_\rho\). As \(0\leqslant \zeta \circ Q_\rho \leqslant 1\) in \(\R^n\) and \(\zeta \circ Q_\rho=0\) in \(B_{3/4}\), from~\eqref{ZdAlT} we have that \begin{align*}
\vert (-\Delta)^s (\zeta \circ Q_\rho)(x) \vert &=C \int_{B_{3/4}(-2e_1/\rho)} \frac{(\zeta \circ Q_\rho)(y)}{\vert x - y\vert^{n+2s}} \dd y \leqslant C \qquad \text{for all }x \in B_{3/4} 
\end{align*} using also
that \(\vert x- y\vert \geqslant 2(1/\rho -3/4) \geqslant 1/2\). Moreover, \begin{align*}
\|  (-\Delta)^s \zeta \|_{L^\infty(B_{3/4})} &\leqslant C ( \| D^2\zeta \|_{L^\infty(B_{3/4})} + \| \zeta \|_{L^\infty(\R^n)} ) \leqslant C, 
\end{align*} for example, see the computation on p.~9 of \cite{MR3469920}. Thus,
$$\| (-\Delta)^s \varphi^{(1)} \|_{L^\infty(B_{3/4})} \leqslant \| (-\Delta)^s \zeta\|_{L^\infty(B_{3/4})} + \| (-\Delta)^s (\zeta \circ Q_\rho)\|_{L^\infty(B_{3/4})} \leqslant C,$$
which completes the proof. 
\end{proof}

Now we give the proof of \thref{YLj1r}. 

\begin{proof}[Proof of \thref{YLj1r}]
Let \(\tilde u(x):=u (\rho x+e_1)\) and \(\tilde c(x) := \rho^{2s}c(\rho x+e_1)\). Observe that \((-\Delta)^s\tilde u+\tilde c \tilde u \geqslant - M\rho^{2s}\) in \(B_1\) and that \(\tilde u \) is antisymmetric with respect to \(T_\rho\).

Let \(\varphi^{(1)}\) be defined as in \thref{rvZJdN88} and suppose that\footnote{Note that such a \(\tau\) exists. Indeed, let \(U \subset \R^n\), \(u:U \to \R\) be a nonnegative function and let \(\varphi : U \to \R\) be such that there exists \(x_0\in U\) for which \(\varphi(x_0)>0\). Then define \begin{align*}
I := \{ \tau \geqslant 0 \text{ s.t. } u(x) \geqslant \tau \varphi(x) \text{ for all } x\in U \} .
\end{align*} It follows that \(I\) is non-empty since \(0\in I\), so \(\sup I\) exists, but may be \(+\infty\). Moreover, \(\sup I \leqslant u(x_0)/\varphi(x_0) < +\infty\), so \(\sup I\) is also finite.} \(\tau\geqslant 0\) is the largest possible value such that~\(\tilde u  \geqslant \tau \varphi^{(1)} \) in the half space~\(H_\rho\). Since \(\varphi^{(1)} =1\) in \(B_{1/2}\), we immediately obtain that \begin{align}
\tau \leqslant \inf_{B_{1/2}} \tilde u = \inf_{B_{\rho/2}(e_1)} u. \label{mzMEg}
\end{align}

Moreover, by continuity, there exists \(a \in \overline{B_{3/4}}\) such that \(\tilde u(a)=\tau\varphi^{(1)} (a)\). On one hand, using \thref{rvZJdN88}, we have that \begin{align}
(-\Delta)^s(\tilde u-\tau \varphi^{(1)})(a)+\tilde c (a) (\tilde u-\tau \varphi^{(1)})(a) &\geqslant -M\rho^{2s} - \tau \big(  C+ \|\tilde c\|_{L^\infty(B_1)}\big) \nonumber \\
&\geqslant -M\rho^{2s} - C\tau \big(  1+ \rho^{2s}\|c\|_{L^\infty(B_\rho(e_1))}\big) .\label{aDLwG3pX}
\end{align} On the other hand, since \(\tilde u-\tau\varphi^{(1)}\) is antisymmetric with respect to \(T_\rho\), \(\tilde u - \tau \varphi^{(1)} \geqslant 0\) in \(H_\rho\), and~\((\tilde u-\tau \varphi^{(1)})(a)=0\), it follows that \begin{align}
&(-\Delta)^s (\tilde u-\tau \varphi^{(1)})(a)+c_\rho(a)(\tilde u-\tau \varphi^{(1)})(a) \nonumber \\
&\leqslant - C \int_{B_{1/2}} \left(
\frac 1 {\vert a - y \vert^{n+2s}} - \frac 1 {\vert Q_\rho(a) - y \vert^{n+2s}} \right) \big(
\tilde u(y)-\tau \varphi^{(1)}(y)\big) \dd y. \label{wPcy8znV}
\end{align} For all \(y \in B_{1/2}\), we have that \(\vert a - y \vert \leqslant C\) and \(\vert Q_\rho(a) - y \vert \geqslant C\) (the assumption \(\rho <1\) allows to choose this \(C\) independent of \(\rho\)), so \begin{align}
(-\Delta)^s (\tilde u-\tau \varphi^{(1)})(a)+c_\rho(a)(\tilde u-\tau \varphi^{(1)})(a) &\leqslant - C \int_{B_{1/2}}  \big(\tilde u(y)-\tau \varphi^{(1)}(y)\big) \dd y \nonumber \\
&\leqslant -C \left( \frac 1 {\rho^n} \int_{B_{\rho/2}(e_1)}  u(y) \dd y - \tau \right). \label{ETn5BCO5}
\end{align} Rearranging \eqref{aDLwG3pX} and \eqref{ETn5BCO5} then using \eqref{mzMEg}, we obtain \begin{align*}
\frac 1 {\rho^n} \int_{B_{\rho/2}(e_1)}  u(y) \dd y &\leqslant C  \Big( \tau \big(  1+ \rho^{2s}\|c\|_{L^\infty(B_\rho(e_1))}\big)  + M\rho^{2s} \Big) \\
&\leqslant C \big(  1+ \rho^{2s}\|c\|_{L^\infty(B_\rho(e_1))}\big) \left(
\inf_{B_{\rho/2}(e_1)} u + M \rho^{2s}\right)
\end{align*} as required.
\end{proof}

A simple adaptation of the proof of \thref{YLj1r} leads to the proof of \thref{MT9uf} which we now give. 

\begin{proof}[Proof of \thref{MT9uf}]
Follow the proof of \thref{YLj1r} but instead of \eqref{wPcy8znV}, we write \begin{align*}
&(-\Delta)^s (\tilde u-\tau \varphi^{(1)})(a)+c_\rho(a)(\tilde u-\tau \varphi^{(1)})(a)\\
&\qquad= - C \int_{H_\rho} \left(
\frac 1 {\vert a - y \vert^{n+2s}} - \frac 1 {\vert Q_\rho(a) - y \vert^{n+2s}} \right) \big(
\tilde u(y)-\tau \varphi^{(1)}(y)\big) \dd y.
\end{align*} Then, for all \(x,y\in H_\rho\), \begin{align}
\frac 1 {\vert x - y \vert^{n+2s}} - \frac 1 {\vert Q_\rho(x) - y \vert^{n+2s}} &= \frac{n+2s} 2 \int_{\vert x-y\vert^2}^{\vert Q_\rho(x)-y\vert^2} t^{- \frac{n+2s+2}2} \dd t \label{gvKkSL1N}\\
&\geqslant C \Big (\vert Q_\rho(x)-y\vert^2-\vert x-y\vert^2 \Big ) \vert Q_\rho(x) - y \vert^{- (n+2s+2)}  \nonumber\\
&= C \frac{(x_1+1/\rho)(y_1+1/\rho)}{\vert Q_\rho(x) - y \vert^{n+2s+2}}, \nonumber
\end{align} so using that \(\tilde u  - \tau \varphi^{(1)} \geqslant 0\) in \(H_\rho\), we see that \begin{align*}
(-\Delta)^s (\tilde u-\tau \varphi^{(1)})(a)+c_\rho(a)(\tilde u-\tau \varphi^{(1)})(a) 
&\leqslant - C \int_{H_\rho} \frac{(y_1+1/\rho)(\tilde u(y)-\tau \varphi^{(1)}(y))  }{\vert Q_\rho(a) - y \vert^{n+2s+2}}\dd y \\
&\leqslant - C_\rho  \left( \int_{H_\rho} \frac{(y_1+1/\rho)\tilde u(y) }{\vert Q_\rho(a) - y \vert^{n+2s+2}}\dd y + \tau \right) .
\end{align*} Making the change of variables \(z=\rho y +e_1\), we have that \begin{align*}
\int_{H_\rho} \frac{(y_1+1/\rho)\tilde u(y) }{\vert Q_\rho(a) - y \vert^{n+2s+2}}\dd y  &= \rho^{-n-1} \int_{\R^n_+} \frac{z_1 u(z) }{\vert Q_\rho(a) - z/\rho +e_1/\rho \vert^{n+2s+2}}\dd z.
\end{align*} Thus, since \( \vert Q_\rho(a) - z/\rho +e_1/\rho \vert^{n+2s+2} \leqslant C_\rho (1+ \vert z \vert^{n+2s+2} ) \), we conclude that
\begin{align*}
\int_{H_\rho} \frac{(y_1+1/\rho)\tilde u(y) }{\vert Q_\rho(a) - y \vert^{n+2s+2}}\dd y &\geqslant C_\rho \int_{\R^n_+} \frac{z_1 u(z) }{1+ \vert z \vert^{n+2s+2}} \dd z=C_\rho \Anorm{u}.
\end{align*} As a consequence,
\begin{align}
(-\Delta)^s (\tilde u-\tau \varphi^{(1)})(a)+\tilde c(a)(\tilde u-\tau \varphi^{(1)})(a) &\leqslant- C_\rho \Anorm{u}+C_\rho \tau  . \label{Pzf2a}
\end{align} Rearranging~\eqref{aDLwG3pX} and~\eqref{Pzf2a} then using~\eqref{mzMEg} gives \begin{align*}
\Anorm{u} &\leqslant C_\rho ( \tau + M ) \leqslant C_\rho \left( \inf_{B_{\rho/2}(e_1)} u + M \right), 
\end{align*}
as desired.
\end{proof}

\subsection{Interior local boundedness}

The second part of the proof of \thref{DYcYH} is the interior local boundedness of sub-solutions given in \thref{guDQ7} below. 

\begin{prop} \thlabel{guDQ7}
Let \(M \geqslant 0\), \(\rho\in(0,1/2)\), and \(c\in L^\infty(B_\rho(e_1))\). Suppose that \(u \in C^{2s+\alpha}(B_\rho(e_1))\cap \mathscr A_s(\R^n)\) for some \(\alpha>0\) with \(2s+\alpha\) not an integer, and \(u\) satisfies   \begin{equation}\label{sow85bv984dert57nb5}
(-\Delta)^su +cu \leqslant M \qquad \text{in } B_\rho(e_1).
\end{equation} 

Then there exists \(C_\rho>0\) depending only on \(n\), \(s\), \(\| c \|_{L^\infty(B_\rho(e_1))}\), and \(\rho\) such that  \begin{align*}
\sup_{B_{\rho/2}(e_1)} u &\leqslant C_\rho ( \Anorm{u} +M  ) .
\end{align*} 
\end{prop} 

The proof of \thref{guDQ7} uses similar ideas to \cite[Theorem~11.1]{MR2494809} and \cite[Theorem~5.1]{MR2831115}. Before we prove \thref{guDQ7}, we need the following lemma. 

\begin{lem} \thlabel{ltKO2}
Let \(R\in(0,1)\) and \(a\in \overline{ B_2^+}\). Then there exists \(C>0\) depending only on \(n\) and \(s\) such that, for all \(x\in B_{R/2}^+(a)\) and \(y \in \R^n_+ \setminus B_R^+(a)\), \begin{align*}
\frac 1 {\vert x - y \vert^{n+2s}} - \frac 1 {\vert x_\ast - y \vert^{n+2s}} &\leqslant C R^{-n-2s-2} \frac{x_1y_1}{1+\vert y \vert^{n+2s+2}}.
\end{align*} 
\end{lem}

\begin{proof} Let \(\tilde x := (x-a)/R\in B_{1/2}\) and \(\tilde y = (y-a)/ R\in \R^n\setminus B_1\). Clearly,
we have that~\(\vert \tilde y - \tilde x \vert \geqslant 1/2\). Moreover,  since \(\vert \tilde x \vert <1/2 < 1/(2 \vert \tilde y \vert)\), we have that \(\vert \tilde y - \tilde x \vert \geqslant \vert \tilde y \vert -\vert \tilde x \vert  \geqslant (1/2) \vert \tilde y \vert\). Hence, \begin{align*}
\vert \tilde y - \tilde x \vert^{n+2s+2} \geqslant \frac 1 {2^{n+2s+2}}  \max  \big \{ 1 , \vert \tilde y \vert^{n+2s+2} \big \} \geqslant C \big ( 1+  \vert \tilde y \vert^{n+2s+2} \big ) 
\end{align*} for some \(C\) depending only on \(n\) and \(s\). It follows that
\begin{align}
\vert x -y \vert^{n+2s+2} = R^{n+2s+2}  \vert \tilde x - \tilde y  \vert^{n+2s+2} 
\geqslant C R^{n+2s+2} \big ( 1 + R^{-n-2s-2} \vert y-a \vert^{n+2s+2} \big ). \label{TO9DRBuX}
\end{align}

Finally, we claim that there exists \(C\) independent of \(R\) such that \begin{align}
\vert y - a\vert \geqslant CR \vert y \vert \qquad \text{for all } y\in \R^n \setminus B_R(a) . \label{tdNyAxBN}
\end{align} Indeed, if \(y \in \R^n \setminus B_4\) then \begin{align*}
\vert y - a\vert \geqslant \vert y \vert -2 \geqslant \frac12 \vert y \vert>\frac R 2 \vert y \vert,
\end{align*} and if \(y \in (\R^n \setminus B_R(a) ) \cap B_4\) then \begin{align*}
\vert y -a\vert \geqslant R > \frac R 4 \vert y \vert
\end{align*} which proves \eqref{tdNyAxBN}.

Thus, \eqref{TO9DRBuX} and \eqref{tdNyAxBN} give that, for all \( x \in B_{R/2}(a)\) and \( y \in \R^n_+ \setminus B_R(a)\), \begin{align*}
\vert x -y \vert^{n+2s+2} &\geqslant C R^{n+2s+2} \big ( 1 +  \vert y\vert^{n+2s+2} \big ) .
\end{align*} Then the result follows directly from~\eqref{LxZU6}. 
\end{proof}

With this preliminary work, we now focus on the proof of \thref{guDQ7}.

\begin{proof}[Proof of \thref{guDQ7}]
We first observe that,
dividing through by~\( \Anorm{u} +M\), we may
assume that \((-\Delta)^su +cu\leqslant 1\) in \(B_\rho(e_1)\) and~\(\Anorm{u} \leqslant 1\). 

We also point out that if~$u\le0$ in~$B_\rho(e_1)$, then the claim in \thref{guDQ7}
is obviously true. Therefore, we can suppose that
\begin{equation}\label{upos44567890}
\{u>0\}\cap B_\rho(e_1)\ne \varnothing.\end{equation}
Thus, we let~\(\tau\geqslant 0 \) be the smallest possible value such that \begin{align*}
u(x) &\leqslant \tau (\rho-\vert x-e_1 \vert )^{-n-2} \qquad \text{for all } x\in B_\rho(e_1).
\end{align*} Such a \(\tau\) exists in light of~\eqref{upos44567890} by following a similar argument to the one in the footnote at the bottom of p. 11.

To complete the proof, we will show that \begin{align}
\tau &\leqslant C_\rho  \label{svLyO}
\end{align}with \(C_\rho\) depending only on \(n\), \(s\), \(\| c\|_{L^\infty(B_\rho(e_1))}\),
and \(\rho\) (but independent of \(u\)). Since \(u\) is uniformly
continuous in~$B_\rho(e_1)$, there exists \(a\in B_\rho(e_1)\) such that \(u(a) = \tau (\rho-\vert a-e_1\vert)^{-n-2}\).
Notice that~$u(a)>0$ and~\(\tau >0\). 

Let also~\(d=:\rho-\vert a-e_1\vert\) so that \begin{align}
u(a) = \tau d^{-n-2} ,\label{CEnkR}
\end{align}  and let  \begin{align*}
U:= \bigg \{ y \in B_\rho (e_1) \text{ s.t. } u(y) > \frac{u(a)} 2 \bigg \} .
\end{align*}  Since \(\Anorm{u}\leqslant 1\),  if \(r\in(0,d)\) then \begin{align*}
C_\rho &\geqslant  \int_{B_\rho (e_1)} |u(x)| \dd x
\ge \int_{ U\cap B_r (a)} u(x) \dd x \geqslant \frac{u(a)}2 \cdot \vert U \cap B_r (a) \vert  .
\end{align*} Thus, from~\eqref{CEnkR}, it follows that  \begin{align}
 \vert  U \cap B_r (a) \vert \leqslant \frac{C_\rho d^{n+2}}\tau \qquad \text{ for all } r\in(0,d). \label{W5Ar0}
\end{align}

Next, we make the following claim. 

\begin{claim} 
There exists~$\theta_0\in(0,1)$ depending only on \(n\), \(s\), \(\| c\|_{L^\infty(B_\rho(e_1))}\), and \(\rho\)
such that if~\(\theta\in(0,\theta_0]\) there exists~\(C>0\) depending only on \(n\), \(s\), \(\| c\|_{L^\infty(B_\rho(e_1))}\), \(\rho\), and~$\theta$ such that \begin{align*}
\big \vert B_{\theta d /8} (a)  \setminus U \big \vert \leqslant  \frac 1 4 \vert B_{\theta d /8} \vert +C \frac{d^n }{\tau} . 
\end{align*} In particular, neither \(\theta\) nor \(C\) depend on \(\tau\), \(u\), or \(a\).
\end{claim} 

We will withhold the proof of the claim until the end. Assuming the claim is true, we complete the proof of the \thref{guDQ7} as follows. By~\eqref{W5Ar0} (used here with~$r:={\theta_0 d}/8$)
and the claim, we have that \begin{align*}
\frac{C_\rho d^{n+2}}\tau  \geqslant \big \vert B_{\theta_0 d /8}  \big \vert -\big \vert B_{\theta_0 d /8} (a)  \setminus U \big \vert 
\geqslant \frac 3 4 \vert B_{\theta_0 d/8} \vert -C \frac{d^n }{\tau} .
\end{align*} Rearranging gives that \( \tau \leqslant C( d^2 +1 ) \leqslant C\), which proves \eqref{svLyO}. 

Accordingly, to complete the proof of \thref{guDQ7}, it remains to establish
 the Claim. For this, let \(\theta\in(0,1)\) be a small constant to be chosen later. We will prove the claim by applying \thref{YLj1r} to an appropriate auxiliary function. For \(x\in B_{\theta d/2} (a)\), we have that \(\vert x -e_1 \vert \leqslant \vert a -e_1 \vert + \theta d/2=\rho-(1-\theta/2)d \), so, using~\eqref{CEnkR}, \begin{align}
u(x) \leqslant \tau \bigg (1-\frac \theta2\bigg )^{-n-2} d^{-n-2} =  \bigg (1-\frac\theta 2\bigg )^{-n-2} u(a) \qquad \text{for all } x\in B_{\theta d/2}(a). \label{sS0kO}
\end{align}

Let \(\zeta \) be a smooth, antisymmetric function such that~$\zeta \equiv 1$ in~$\{x_1 > 1/2 \}$
and~$0\leqslant \zeta \leqslant 1$ in~$\R^n_+$, 
and consider the antisymmetric function \begin{align*}
v(x) &:=   \left(1-\frac\theta 2\right)^{-n-2} u(a) \zeta(x) -u (x) \qquad \text{for all } x\in \R^n. 
\end{align*} Since \(\zeta \equiv 1\) in \(\{x_1 > 1/2 \} \supset B_{\theta d/2}(a)\) and \(0\leqslant \zeta \leqslant 1\) in \(\R^n_+\), it follows easily from~\eqref{OopaQ0tu} that \((-\Delta)^s \zeta \geqslant 0\) in \(B_{\theta d/2}(a)\). Hence, in \(B_{\theta d/2}(a)\), \begin{align*}
(-\Delta)^s v +cv &\geqslant -(-\Delta)^s u-cu+  c\bigg (1-\frac\theta 2\bigg )^{-n-2} u(a)  \\
&\geqslant -1 -C\| c^-\|_{L^\infty(B_\rho(e_1))}\bigg (1-\frac\theta 2\bigg )^{-n-2}  u(a) . 
\end{align*} Taking \(\theta\) sufficiently small, we obtain \begin{align}
(-\Delta)^s v +cv &\geqslant -C (1 +  u(a) ) \qquad \text{in } B_{\theta d/2}(a). \label{7fUyX}
\end{align}

The function \(v\) is almost the auxiliary function to which we would like to apply \thref{YLj1r}; however, \thref{YLj1r} requires \(v \geqslant 0\) in \(\R^n_+\) but we only have \(v \geqslant 0\) in \(B_{\theta d/2}(a)\) due to \eqref{sS0kO}. To resolve this issue let us instead consider the function \(w\) such that \(w (x) = v^+(x)\) for all \(x\in \R^n_+\) and~\(w(x) = -w(x_\ast)\) for all \(x\in \overline{\R^n_-}\).
We point out that~$w$ coincides with~$v$ in~$B_{\theta d/2}(a)$, thanks to~\eqref{sS0kO},
and therefore it is as regular as~$v$ in~$B_{\theta d/2}(a)$, which allows us to write the fractional Laplacian of~$w$ in~$B_{\theta d/2}(a)$ in a pointwise sense.

Also we observe that we have~\(w\geqslant0\) in \(\R^n_+\) but we no longer have a lower bound for \((-\Delta)^sw+cw\). To obtain this, observe that for all \(x\in \R^n_+\), \begin{align*}
(w-v)(x) 
&= \begin{cases}
0, &\text{for all } x\in \{v>0\} \cap \R^n_+ \\
u(x)- (1-\theta /2 )^{-n-2} u(a) \zeta(x), &\text{for all } x\in \{v \leqslant 0\}\cap \R^n_+. 
\end{cases}
\end{align*} In particular, \(w-v\leqslant |u|\) in \(\R^n_+\). It follows that for all \(x\in B_{\theta d/2}(a)\),
\begin{align*}
(-\Delta)^s (w-v)(x) &\geqslant -C \int_{\R^n_+ \setminus B_{\theta d/2}(a) } \left( \frac 1 {\vert x-y\vert^{n+2s}} - \frac 1 {\vert x_\ast - y \vert^{n+2s}} \right)| u(y)|\dd y.
\end{align*} Moreover, by \thref{ltKO2}, for all~$x\in B_{\theta d/4}(a)$,\begin{align}
(-\Delta)^s (w-v)(x) &\geqslant -C (\theta d)^{-n-2s-2} \Anorm{u}
\geqslant -C (\theta d)^{-n-2s-2}  . \label{Fy8oddTC} 
\end{align}Thus, by~\eqref{7fUyX} and~\eqref{Fy8oddTC}, for all \(x\in  B_{\theta d/4}(a)\), we have that \begin{align}
(-\Delta)^sw(x)+c(x)w(x) &= (-\Delta )^s v (x)+c(x) v(x) + (-\Delta)^s (w-v)(x)
\nonumber \\ 
&\geqslant -C \big (1 +  u(a) +(\theta d)^{-n-2s-2} \big )  \nonumber \\
&\geqslant  -C \big ( (\theta d)^{-n-2s-2}+u(a)  \big ) \label{c1kcbe0C}
\end{align} using that \(\theta d<1\).

Next let us consider the rescaled and translated functions \(\tilde w(x) := w(a_1 x+(0,a'))\) and \(\tilde c (x) := a_1^{2s} c(a_1 x+(0,a'))\) (recall that \(a' = (a_2,\dots, a_n) \in \R^{n-1}\)). By~\eqref{c1kcbe0C} we have that
\begin{align*}
(-\Delta)^s \tilde w+\tilde c\,  \tilde w  \geqslant -Ca_1^{2s} \big ((\theta d)^{-n-2s-2} +  u(a) \big )  \qquad \text{in } B_{\theta d /(4a_1)}(e_1 ).
\end{align*}  On one hand, by \thref{YLj1r}, we obtain \begin{align*}
\left(\frac{\theta d } 8 \right)^{-n} \int_{B_{\theta d /8}(a )}  w( x) \dd x 
&= \left(\frac{\theta d }{8a_1} \right)^{-n} \int_{B_{\theta d/(8a_1)}(e_1 )} \tilde w ( x) \dd x \\
&\leqslant C \Big(  \tilde w(e_1) +  (\theta d)^{-n-2} +u(a) (\theta d )^{2s} \Big)\\
&= C  \left( \left(1-\frac\theta 2\right)^{-n-2} -1 \right)u(a)  +C (\theta d)^{-n-2} +u(a) (\theta d )^{2s} .
\end{align*}We note explicitly that, by~\eqref{CyJJQHrF}, the constant in the above line is given by \begin{align*}
C= C'\big ( 1 +(\theta d)^{2s} \| c\|_{L^\infty(B_{\rho}(e_1))} \big ),
\end{align*} so using that \(\theta d<1\) it may be chosen to depend only on \(n\), \(s\), and \( \| c\|_{L^\infty(B_{\rho}(e_1))}\). On the other hand, \begin{align*}
B_{\theta d /8} (a)  \setminus U  &\subseteq\left\{ w \geqslant \left( \left( 1- \frac \theta  2\right)^{-n-2} -\frac12 \right) u(a) \right\} \cap B_{\theta d /8} (a)  ,
\end{align*}so we have that \begin{align*}
(\theta d )^{-n} \int_{B_{\theta d/8}(a)} w(x) \dd x  &\geqslant (\theta d)^{-n} \bigg ( \bigg ( 1- \frac \theta  2\bigg )^{-n-2} -\frac12 \bigg ) u(a) \cdot  \big \vert B_{\theta d /8} (a)  \setminus U\big   \vert \\
&\geqslant  C (\theta d)^{-n} u(a) \cdot  \big \vert B_{\theta d /8} (a)  \setminus U\big   \vert 
\end{align*} for \(\theta\) sufficiently small. Thus, \begin{align*}
\big \vert B_{\theta d /8} (a)  \setminus U \big   \vert
&\leqslant C  (\theta d)^n \bigg ( \bigg (1-\frac\theta 2\bigg )^{-n-2} -1 \bigg )  +C(u(a))^{-1} (\theta d)^{-2} + (\theta d )^{n+2s} \\
&\leqslant C  (\theta d)^n \bigg ( \bigg (1-\frac\theta 2\bigg )^{-n-2} -1 +\theta^{2s}\bigg )  +C \frac{\theta^{-2} d^n }{\tau} 
\end{align*} using~\eqref{CEnkR} and that \(d^{n+2s}<d^n\) since \(d<1\). At this point we may choose \(\theta\) sufficiently small such that \begin{align*}
 (\theta d)^n \bigg ( \bigg (1-\frac\theta 2\bigg )^{-n-2} -1 +\theta^{2s}\bigg )  \leqslant \frac 1 4 \vert B_{\theta d/8} \vert .
\end{align*} This proves the claim, and thus completes the proof of
Proposition~\ref{guDQ7}.
\end{proof}

\section{Boundary Harnack inequality and proof of \thref{C35ZH}} \label{TZei6Wd4}

In this section, we give the proof of \thref{C35ZH}. Analogous to the proof of \thref{DYcYH}, the proof
of \thref{C35ZH} is divided into the boundary Harnack inequality for super-solutions (\thref{SwDzJu9i}) and 
the boundary local boundedness for sub-solutions (\thref{EP5Elxbz}). Together these two results immediately 
give \thref{C35ZH}. 

\subsection{The boundary weak Harnack inequality}

Our next result is the antisymmetric boundary weak Harnack inequality. 

\begin{prop} \thlabel{SwDzJu9i} 
Let \(M\in \R\), \(\rho>0\), and \(c\in L^\infty(B_\rho^+)\). Suppose that \(u\in C^{2s+\alpha}(B_\rho)\cap \mathscr{A}_s(\R^n)\) for some \(\alpha > 0\) with \(2s+\alpha\) not an integer, \(u\) is non-negative in \(\R^n_+\) and satisfies \begin{align*}
(-\Delta)^s u +cu \geqslant -Mx_1 \qquad \text{in } B_\rho^+.
\end{align*}

Then there exists \(C_\rho>0\) depending only on \(n\), \(s\), \(\| c \|_{L^\infty(B_\rho^+)}\), and \(\rho\) such that \begin{align*}
\Anorm{u}  &\leqslant  C_\rho \left(  \inf_{B_{\rho/2}^+} \frac{u(x)}{x_1} +M \right).
\end{align*}  
\end{prop} 

As with the interior counter-part of \thref{SwDzJu9i}, that is
\thref{MT9uf}, we will prove the following rescaled version of \thref{SwDzJu9i},
namely \thref{g9foAd2c}. This version is essential to the proof of the boundary local boundedness for sub-solutions. Once \thref{g9foAd2c} has been proven, \thref{SwDzJu9i} follows easily with some minor adjustments. 

\begin{prop} \thlabel{g9foAd2c} 
Let \(M\in \R\), \(\rho >0\), and \(c\in L^\infty(B_\rho^+)\). Suppose that \(u\in C^{2s+\alpha}(B_\rho)\cap \mathscr{A}_s(\R^n)\) for some \(\alpha >0\) with~\(2s+\alpha\) not an integer, \(u\) is non-negative in \(\R^n_+\) and satisfies \begin{align*}
(-\Delta)^s u +cu \geqslant -Mx_1\qquad \text{in } B_\rho^+.
\end{align*}

Then there exists \(C>0\) depending only on \(n\), \(s\), and \(\rho^{2s}\| c \|_{L^\infty(B_\rho^+)}\) such that \begin{align*}
\frac 1 {\rho^{n+2}} \int_{B_{\rho/2}^+} y_1 u(y) \dd y &\leqslant C \left(  \inf_{B_{\rho/2}^+} \frac{u(x)}{x_1} + M \rho^{2s}\right). 
\end{align*}  Moreover, the constant \(C\) is of the form \begin{align*}
C=C' \big(1+\rho^{2s} \| c \|_{L^\infty(B_\rho^+)} \big) 
\end{align*} with \(C'\) depending only on \(n\) and \(s\). 
\end{prop}

Before we prove \thref{g9foAd2c}, we require some lemmata. 

\begin{lem} \thlabel{tZVUcYJl}
Let \(M\geqslant0\), \(k \geqslant 0\), and suppose that \(u\in C^{2s+\alpha}(B_1)\cap \mathscr{A}_s(\R^n)\) for some~\(\alpha>0\) with~\(2s+\alpha\) not an integer. \begin{enumerate}[(i)]
\item If \(u\) satisfies \begin{align*}
(-\Delta)^s u +k u \geqslant -Mx_1\qquad \text{in } B_1^+
\end{align*} then for all \(\varepsilon>0\) sufficiently small there exists \(u_\varepsilon \in C^\infty_0(\R^n)\)
antisymmetric and such that \begin{equation}
(-\Delta)^s u_\varepsilon +k u_\varepsilon \geqslant -(M+\varepsilon) x_1\qquad \text{in } B_{7/8}^+. \label{EuoJ3En6}
\end{equation}
\item If \(u\) satisfies \begin{align*}
(-\Delta)^s u +k u \leqslant Mx_1\qquad \text{in } B_1^+
\end{align*} then for all \(\varepsilon>0\) sufficiently small there exists \(u_\varepsilon \in C^\infty_0(\R^n)\)
antisymmetric and such that \begin{align*}
(-\Delta)^s u_\varepsilon +k u_\varepsilon \leqslant (M+\varepsilon) x_1\qquad \text{in } B_{7/8}^+. 
\end{align*}
\end{enumerate} In both cases the sequence \(\{ u_\varepsilon\} \) converges to \(u\) uniformly in \(B_{7/8}\).

Additionally, if \(u \) is non-negative in \(\R^n_+\) then \(u_\varepsilon\) is also non-negative in \(\R^n_+\).
\end{lem}

For the usual fractional Laplacian, \thref{tZVUcYJl} follows immediately by taking a mollification of \(u\) and in principle, this is also the idea here. However, there are a couple of technicalities that need to be addressed. The first is that here the fractional Laplacian is defined according to \thref{mkG4iRYH} and it remains to be verified that this fractional Laplacian commutes with the convolution operation as the usual one does. As a matter of fact, \thref{mkG4iRYH} does not lend itself well to the Fourier transform which makes it difficult to prove such a property. We overcome this issue by first multiplying \(u\) by an appropriate cut-off function which allows us to reduce to the case \((-\Delta)^s\) as given by the usual definition.  

The second issue is that directly using the properties of mollifiers, we can only expect to
control~\(u_\varepsilon\) in some~\(U \Subset B_1^+\) and not up to~\(\{x_1=0\}\). We can relax this thanks to the antisymmetry of \(u\). 

\begin{proof}[Proof of Lemma~\ref{tZVUcYJl}]
Fix \(\varepsilon>0\). Let \(R>1\) and let~\(\zeta\) be a smooth radial cut-off function such that $$
\zeta \equiv 1 \text{ in } B_R, \quad \zeta \equiv 0 \text{ in } \R^n \setminus B_{2R},
\quad{\mbox{and}}\quad 0\leqslant \zeta \leqslant 1 .
$$
Let also \(\bar  u := u \zeta\).

Now let us define a function \(f:B_1\to\R\) as follows: let \(f(x)=(-\Delta)^s u(x) + k u(x)\) for all \(x\in B_1^+\), \(f(x) = 0\) for all \(x\in B_1\cap \{x_1=0\}\), and \(f(x)=-f(x_\ast)\). We also define \(\bar f:B_1\to \R\) analogously with \(u\) replaced with \(\bar u\). By definition, both \(f\) and \(\bar f\) are antisymmetric\footnote{Note that the definition of antisymmetric requires the domain of \(f\) and \(\bar f\) to be \(\R^n\). For simplicity, we will still refer to \(f\) and \(\bar f\) as antisymmetric in this context since this technicality does not affect the proof.}, but note carefully that, \emph{a priori}, there is no reason to expect any regularity of \(f\) and \(\bar f\) across \(\{x_1=0\}\) (we will in fact find that \(\bar f \in C^\alpha(B_1)\)).

We claim that for \(R\) large enough (depending on \(\varepsilon\)), \begin{align}
\vert \bar f(x) - f(x) \vert \leqslant \varepsilon x_1\qquad \text{for all } x\in B_1^+. \label{8LncZIti}
\end{align} Indeed, if \(x\in B_1^+\) then \begin{align*}
\bar f(x) - f(x) =(-\Delta)^s(\bar u - u ) (x) &= C \int_{\R^n_+\setminus B_R} \bigg ( \frac 1 {\vert x - y \vert^{n+2s}} - \frac 1 {\vert x_\ast- y \vert^{n+2s}} \bigg ) ( u -\bar u ) (y) \dd y .
\end{align*} {F}rom~\eqref{LxZU6}, it follows that \begin{align*}
\vert \bar f(x) - f(x) \vert  \leqslant  Cx_1 \int_{\R^n_+\setminus B_R}  \frac{y_1\vert u(y) -\bar  u(y)\vert}{\vert x - y \vert^{n+2s+2}} \dd y \leqslant Cx_1 \int_{\R^n_+\setminus B_R}  \frac{y_1\vert u(y) \vert }{1+ \vert y \vert^{n+2s+2}} \dd y .
\end{align*} Since \(u\in \mathscr A_s(\R^n)\), taking \(R\) large we obtain~\eqref{8LncZIti}.

Next, consider the standard mollifier \(\eta(x) := C_0 \chi_{B_1}(x) e^{-\frac 1 {1-\vert x\vert ^2}}\) with \(C_0>0\) such that \(\int_{\R^n} \eta(x) \dd x =1\) and let~\(\eta_\varepsilon(x) := \varepsilon^{-n} \eta (x/\varepsilon)\). Also, let \(u_\varepsilon := \bar u \ast \eta_\varepsilon\) and \(f_\varepsilon := \bar f \ast \eta_\varepsilon\).

Notice that~$u_\epsilon\in C^\infty_0(\R^n)$ and it is antisymmetric. Additionally, we show that~\eqref{EuoJ3En6}
holds true in case~$(i)$ (case~$(ii)$ being analogous).

To this end, we observe that, since \(\bar u\) has compact support, we have that~\(\bar u \in \mathscr L_s(\R^n)\), so by \thref{eeBLRjcZ}, \((-\Delta)^s \bar u\) can be understood in the usual sense in \(B_1\), that is, by~\eqref{ZdAlT}. Moreover,
by~\cite[Propositions~2.4-2.6]{MR2270163}, we have that~\((-\Delta)^s\bar u\in C^\alpha(B_1)\) which gives that \(\bar f \in C^\alpha(B_1)\) and \begin{align*}
(-\Delta)^s\bar u +k \bar  u &= \bar f \qquad \text{in }B_1.
\end{align*} In particular, we may use standard properties of mollifiers to immediately obtain \begin{align*}
(-\Delta)^s u_\varepsilon +k u_\varepsilon &= f_\varepsilon \qquad \text{in }B_{7/8}.
\end{align*} 

Also, since \(\bar f\) is antisymmetric, it follows that \begin{align*}
f_\varepsilon(x) &= \int_{\R^n }  \bar f (y) \eta_\varepsilon (x-y) \dd y = \int_{\R^n_+ } \bar f (y)\big (  \eta_\varepsilon (x-y) - \eta_\varepsilon (x_\ast-y) \big )\dd y.
\end{align*} Observe that, since \(\eta \) is monotone decreasing in the radial direction and \(\vert x- y \vert \leqslant \vert x_\ast - y \vert\) for all~\(x,y\in \R^n_+\),  \begin{align}
\eta_\varepsilon (x-y) - \eta_\varepsilon (x_\ast-y) \geqslant 0 \qquad \text{for all } x,y\in \R^n_+. \label{DPmhap5t}
\end{align} Moreover, by~\eqref{8LncZIti}, we see that~\(\bar f (x) \geqslant -(M+\varepsilon)x_1\) for all \(x\in B_1^+\), so if \(x\in B_{7/8}^+\) and \(\varepsilon>0\) is sufficiently small (independent of \(x\)) then it follows that \begin{equation}\label{fjrehgeruig009887}
f_\varepsilon(x) =\int_{B_\varepsilon^+(x)} \bar f (y)\big (  \eta_\varepsilon (x-y) - \eta_\varepsilon (x_\ast-y) \big )\dd y \geqslant -(M+\varepsilon)\int_{B_\varepsilon^+(x)}  y_1 \big (  \eta_\varepsilon (x-y) - \eta_\varepsilon (x_\ast-y) \big )\dd y.
\end{equation}

Next, we claim that \begin{align}
\int_{B_\varepsilon^+(x)}  y_1 \big (  \eta_\varepsilon (x-y) - \eta_\varepsilon (x_\ast-y) \big )\dd y &\leqslant  x_1. \label{Hp0lBCzB}
\end{align} Indeed, \begin{align*}
\int_{B_\varepsilon^+(x)}  y_1 \big (  \eta_\varepsilon (x-y) - \eta_\varepsilon (x_\ast-y) \big )\dd y &= \int_{B_\varepsilon(x)}  y_1   \eta_\varepsilon (x-y) \dd y \\&-\int_{B_\varepsilon^-(x)}  y_1    \eta_\varepsilon (x-y) \dd y - \int_{B_\varepsilon^+(x)}  y_1 \eta_\varepsilon (x_\ast-y) \dd y \\
&= \int_{B_\varepsilon(x)}  y_1   \eta_\varepsilon (x-y) \dd y  - \int_{B_\varepsilon^+(x)\setminus B_\varepsilon^+(x_\ast)}  y_1   \eta_\varepsilon (x_\ast-y) \dd y \\
&\leqslant \int_{B_\varepsilon(x)}  y_1   \eta_\varepsilon (x-y) \dd y .
\end{align*} Moreover, using that \(z \mapsto z_1 \eta(z)\) is antisymmetric and \(\int_{B_\varepsilon} \eta(z)\dd z =1\), we obtain that \begin{align*}
 \int_{B_\varepsilon(x)}  y_1   \eta_\varepsilon (x-y) \dd y &=  \int_{B_\varepsilon(x)}  (y_1-x_1)   \eta_\varepsilon (x-y) \dd y +x_1  \int_{B_\varepsilon(x)}    \eta_\varepsilon (x-y) \dd y =x_1
\end{align*}  which gives \eqref{Hp0lBCzB}. 

{F}rom~\eqref{fjrehgeruig009887} and~\eqref{Hp0lBCzB}, we obtain that
$$ f_\epsilon(x) \ge -(M+\epsilon)x_1$$
for all \(x\in B_{7/8}^+\), as soon as~$\epsilon$ is taken sufficiently small.
This is the desired
result in~\eqref{EuoJ3En6}.

Finally, it follows immediately from the properties of mollifiers that \( u_\varepsilon \to u \) uniformly in \(B_{7/8}\) as~\(\varepsilon\to 0^+\). Moreover, if \(u \geqslant 0\) in \(\R^n_+\) then from antisymmetry, \begin{align*}
u_\varepsilon (x) = \int_{\R^n_+ } \bar u(y)\big (  \eta_\varepsilon (x-y) - \eta_\varepsilon (x_\ast-y) \big )\dd y \geqslant 0 \qquad \text{for all } x\in \R^n_+
\end{align*} using~\eqref{DPmhap5t}. 
\end{proof}

Our second lemma is as follows. 

\begin{lem} \thlabel{6fD34E6w}
Suppose that \( v\in C^\infty_0(\R^n)\) is an antisymmetric function satifying \(\partial_1 v (0)=0\). Then \begin{align*}
\lim_{h\to 0 }\frac{(-\Delta)^sv(he_1)} h  = -2c_{n,s}(n+2s) \int_{\R^n_+} \frac{y_1 v(y)}{\vert y \vert^{n+2s+2}} \dd y .
\end{align*}
\end{lem}

Note that since \(v\in C^\infty_0(\R^n)\), the fractional Laplacian is given by the usual definition as per \thref{eeBLRjcZ}. 

\begin{proof} [Proof of Lemma~\ref{6fD34E6w}]
We will begin by proving that \begin{align}
\lim_{h\to 0 }\frac{(-\Delta)^sv(he_1)} h = (-\Delta)^s \partial_1 v(0). \label{JMT019eU}
\end{align} For this, consider the difference quotient \begin{align*}
\partial^h_1v(x) := \frac{v(x+he_1)-v(x)}h 
\end{align*} for all \(x\in \R^n\) and~\(\vert h \vert>0\) small. Since \(v\) is antisymmetric, \(v(0)=0\), so \begin{align*}
\frac{2v(h e_1)-v(he_1+y)-v(he_1-y)}h&= 2\partial^h_1v(0)-\partial^h_1v(y)-\partial^h_1v(-y) - \frac{v(y)+v(-y)}h
\end{align*} for all \(y\in \R^n\). Moreover, the function~\(y \mapsto v(y)+v(-y)\) is odd with
respect to~$y'$, and so \begin{align*}
\int_{\R^n} \frac{v(y)+v(-y)}{\vert y \vert^{n+2s}}\dd y&=0.
\end{align*} It follows that \begin{align*}
\frac{(-\Delta)^sv(he_1)} h &= \frac{c_{n,s}} 2 \int_{\R^n } \bigg ( 2\partial^h_1v(0)-\partial^h_1v(y)-\partial^h_1v(-y) - \frac{v(y)+v(-y)}h \bigg ) \frac{\dd y }{\vert y \vert^{n+2s}}\\
&= (-\Delta)^s \partial^h_1v(0) .
\end{align*} {F}rom these considerations and the computation at the top of p. 9 in \cite{MR3469920}, we have that \begin{align*}
\bigg \vert \frac{(-\Delta)^sv(he_1)} h  -(-\Delta)^s \partial_1 v(0) \bigg \vert &= \big \vert (-\Delta)^s ( \partial^h_1v-\partial_1v)(0) \big \vert \\
&\leqslant C \Big( \|\partial^h_1v-\partial_1v\|_{L^\infty(\R^n)} + \|D^2\partial^h_1v-D^2\partial_1v\|_{L^\infty(\R^n)} \Big) .
\end{align*} Then we obtain~\eqref{JMT019eU} by sending \(h\to 0\), using that \(\partial^h_1v \to \partial_1v\) in \(C^\infty_{\mathrm{loc}}(\R^n)\) as \(h\to 0\).

To complete the proof, we use that \(\partial_1v(0)=0\) and integration by parts to obtain \begin{align*}
(-\Delta)^s \partial_1 v(0) &= -c_{n,s} \int_{\R^n} \frac{\partial_1v(y)}{\vert y \vert^{n+2s}} \dd y \\
&= c_{n,s} \int_{\R^n} v(y) \partial_1\vert y \vert^{-n-2s} \dd y \\
&= -c_{n,s} (n+2s) \int_{\R^n} \frac{y_1v(y)}{\vert y \vert^{n+2s+2}} \dd y \\
&= -2c_{n,s} (n+2s) \int_{\R^n_+} \frac{y_1v(y)}{\vert y \vert^{n+2s+2}} \dd y 
\end{align*} where the last equality follows from antisymmetry of \(v\). 
\end{proof}

We are now able to give the proof of \thref{g9foAd2c}.

\begin{proof}[Proof of \thref{g9foAd2c}]  Since \(u\) is non-negative in \(\R^n_+\),
we have that~\( (-\Delta)^s u + \| c\|_{L^\infty(B_\rho^+)} u \geqslant  -Mx_1\) in \(B_\rho^+\).
Define~\(\tilde u(x):= u(\rho x)\) and note that \begin{align*}
(-\Delta)^s \tilde u +\rho^{2s} \| c\|_{L^\infty(B_\rho^+)} \tilde  u \geqslant -M\rho^{2s+1} x_1 \qquad  \text{in } B_1^+.
\end{align*} By way of \thref{tZVUcYJl} (i), we may take a \(C^\infty_0(\R^n)\) sequence of functions approximating \(\tilde u\) which satisfy the assumptions of \thref{g9foAd2c} with \(M\) replaced with \(M+\varepsilon\), obtain the estimate, then pass to the limit. In this way we may assume \(\tilde u \in C^\infty_0(\R^n)\). 

Let \(\zeta\) be a smooth radially symmetric cut-off function such that\begin{align*}
\zeta \equiv 1 \text{ in } B_{1/2}, \quad \zeta \equiv 0 \text{ in } \R^n \setminus B_{3/4},\quad {\mbox{and}}
\quad 0\leqslant \zeta \leqslant 1,
\end{align*} and define \(\varphi^{(2)} \in C^\infty_0(\R^n)\) by \(\varphi^{(2)}(x):= x_1 \zeta (x)\) for all \(x\in \R^n\). Suppose that \(\tau \geqslant 0\) is the largest possible value such that \(\tilde u  \geqslant \tau \varphi^{(2)}\) in \(\R^n_+\). For more detail on the existance of such a \(\tau\), see the footnote at the bottom of p. 11. Since \(\varphi^{(2)}(x)=x_1\) in \(B_{1/2}\), we have that \(x_1 \tau \leqslant \tilde u(x)\) for all~\(x\in B_{1/2}\), so \begin{align}
\tau \leqslant \inf_{B_{1/2}^+} \frac{\tilde u(x)}{x_1} =  \rho \inf_{B_{\rho/2}^+} \frac{u(x)}{x_1}. \label{pUJA2JZI}
\end{align} Since \(\tilde u\) is \(C^1\) in \(B_1\), there are two possibilities that can occur: either
there exists \(a\in B_{3/4}^+\) such that~\( \tilde u(a) = \tau \varphi^{(2)} (a)\); or there exists \(a \in B_{3/4} \cap \{ x_1=0\}\) such that \(\partial_1 \tilde u(a) = \tau \partial_1 \varphi^{(2)} (a)\). 

First suppose that there exists \(a\in B_{3/4}^+\) such that \( \tilde u(a) = \tau \varphi^{(2)} (a)\). Since \(\varphi^{(2)}\in C^\infty_0(\R^n)\) and is antisymmetric, \((-\Delta)^s\varphi^{(2)}\) is antisymmetric and~$\partial_1 \varphi^{(2)}=0$ in~$\{x_1=0\}$, we can
exploit Lemma~\ref{6fD34E6w} to say that~\((-\Delta)^s\varphi^{(2)}(x)/x_1\) is bounded in \(\R^n\).

On one hand, using that \((\tilde u-\tau \varphi^{(2)})(a)=0\), we have that \begin{align}
(-\Delta)^s(\tilde u-\tau \varphi^{(2)})(a)  &= (-\Delta)^s(\tilde u-\tau \varphi^{(2)})(a) +\rho^{2s} \| c\|_{L^\infty(B_\rho^+)}
(\tilde u-\tau \varphi^{(2)})(a) \nonumber \\
&\geqslant -M\rho^{2s+1}a_1 -\tau \big  (C + \rho^{2s}\| c \|_{L^\infty(B_\rho^+)} \big  ) a_1 \nonumber \\
&\geqslant  -M\rho^{2s+1}a_1 -C \tau  \big (1 + \rho^{2s} \| c \|_{L^\infty(B_\rho^+)} \big ) a_1 . \label{doW9AF3Y}
\end{align} On the other hand, since \(\tilde u-\tau \varphi^{(2)}\) is antisymmetric, non-negative in \(\R^n_+\), and \((\tilde u-\tau \varphi^{(2)})(a) = 0\), we have by \thref{mkG4iRYH} that  \begin{align*}
(-\Delta)^s(\tilde u-\tau \varphi^{(2)})(a) = -C  \int_{\R^n_+} \bigg (\frac 1 {\vert a - y \vert^{n+2s}} - \frac 1 {\vert a_\ast- y \vert^{n+2s}}\bigg )(\tilde u-\tau \varphi^{(2)})(y) \dd y .
\end{align*} It follows from~\eqref{buKHzlE6} that \begin{align}
(-\Delta)^s(\tilde u-\tau \varphi^{(2)})(a) &\leqslant -C a_1 \int_{B_{1/2}^+}  \frac{y_1 (\tilde u-\tau \varphi^{(2)})(y)} {\vert a_\ast- y \vert^{n+2s+2}}\dd y \nonumber \\
&\leqslant -C a_1 \bigg ( \int_{B_{1/2}^+} y_1 \tilde u(y) \dd y -\tau \bigg ) \nonumber  \\
&=-C a_1 \bigg ( \frac 1 {\rho^{n+1}} \int_{B_{\rho/2}^+} y_1 u(y) \dd y -\tau \bigg ). \label{retlxLlR}
\end{align} Rearranging~\eqref{doW9AF3Y} and~\eqref{retlxLlR}, and recalling~\eqref{pUJA2JZI} gives \begin{align}
\frac 1 {\rho^{n+1}} \int_{B_{\rho/2}^+} y_1 u(y) \dd y &\leqslant C \Big(
\tau (1 + \rho^{2s} \| c \|_{L^\infty(B_\rho^+)} )  + \rho^{2s+1}M \Big) \nonumber  \\
&\leqslant C \rho (1 + \rho^{2s} \| c \|_{L^\infty(B_\rho^+)} )  \left(  \inf_{B_{\rho/2}^+} \frac{u(x)}{x_1} + M \rho^{2s}\right), \label{7O4TL1vF}
\end{align}
which gives the desired result in this case.

Now suppose that there exists \(a \in B_{3/4} \cap \{ x_1=0\}\) such that \(\partial_1 \tilde u(a) =\tau \partial_1 \varphi^{(2)} (a)\). Let \(h>0\) be small and set \(a^{(h)}:= a+ he_1\). On one hand, as in~\eqref{doW9AF3Y},  \begin{align*}
(-\Delta)^s(\tilde u-\tau \varphi^{(2)})(a^{(h)}) &\geqslant -M\rho^{2s+1}h -C \tau \big ( 1+ \rho^{2s}\|c\|_{L^\infty(B_\rho^+)}  \big )h .
\end{align*} Dividing both sides by \(h\) and sending \(h\to 0^+\), it follows from \thref{6fD34E6w} (after a translation) that \begin{align*}
-M\rho^{2s+1} -C \tau \big ( 1+ \rho^{2s}\|c\|_{L^\infty(B_\rho^+)}  \big ) &\leqslant  -C \int_{\R^n_+} \frac{y_1 (\tilde u-\tau \varphi^{(2)})(y)}{\vert y-a\vert^{n+2s+2}} \dd y \\
&\leqslant-C \bigg (  \frac 1{\rho^{n+1}} \int_{B_{\rho/2}^+} y_1 u(y) \dd y - \tau \bigg ). 
\end{align*} Rearranging as before gives the desired result. 
\end{proof}

Next, we give the proof of~\thref{SwDzJu9i}. 

\begin{proof}[Proof of~\thref{SwDzJu9i}]
Let \(\varphi^{(2)}\), \(\tau\), and \(a\) be the same as in the proof of \thref{g9foAd2c}. The proof of \thref{SwDzJu9i} is identical to the proof of \thref{g9foAd2c} except for the following changes.

In the case \(a\in B_{3/4}^+\) and \(\tilde u(a) = \tau \varphi^{(2)}(a)\), we use~\eqref{buKHzlE6} to obtain \begin{align*}
(-\Delta)^s (\tilde u - \tau \varphi^{(2)})(a) \leqslant-C a_1  \int_{\R^n_+} \frac{y_1(\tilde u-\tau \varphi^{(2)})(y)}{\vert a_\ast - y \vert^{n+2s+2}} \dd y \leqslant - C_\rho a_1 \big ( \Anorm{u} -\tau \big )
\end{align*} where we have also used that \begin{align}
\vert a_\ast - y \vert^{n+2s+2} \leqslant C (1+\vert y \vert^{n+2s+2} \big ) \qquad \text{for all } y\in \R^n_+. \label{BWqeWX33}
\end{align} Moreover, in the case \(a\in B_{3/4}\cap \{x_1=0\}\) and \(\partial_1\tilde u(a) = \tau \partial_1\varphi^{(2)}(a)\), 
we have that~\(a=a_\ast\) so~\eqref{BWqeWX33} also gives \begin{align*}
\int_{\R^n_+} \frac{y_1 (\tilde u-\tau \varphi^{(2)})(y)}{\vert y-a\vert^{n+2s+2}} \dd y &\geqslant   C_\rho \big (\Anorm{u} -\tau \big ) .
\qedhere
\end{align*}
\end{proof}

\subsection{Boundary local boundedness}

We now prove the boundary local boundedness for sub-solutions. 

\begin{prop} \thlabel{EP5Elxbz}
Let \(M \geqslant 0\), \(\rho\in(0,1)\), and \(c\in L^\infty(B_\rho^+)\). Suppose that \(u \in C^{2s+\alpha}(B_\rho)\cap \mathscr A_s (\R^n)\) for some \(\alpha>0\) with \(2s+\alpha\) not an integer, and \(u\) satisfies  \begin{align*}
(-\Delta)^su +cu &\leqslant M x_1 \qquad \text{in } B_\rho^+.
\end{align*}

Then there exists \(C_\rho>0\) depending only on \(n\), \(s\), \(\| c \|_{L^\infty(B_\rho^+)}\), and \(\rho\) such that  \begin{align*}
\sup_{x\in B_{\rho/2}^+} \frac{ u(x)}{x_1} &\leqslant C_\rho ( \Anorm{u} +M  ) .
\end{align*} 
\end{prop}

Before we prove \thref{EP5Elxbz}, we prove the following lemma. 

\begin{lem} \thlabel{KaMmndyO}
Let \(\varphi \in C^\infty(\R)\) be an odd function such that \(\varphi(t)=1\) if \(t>2\) and \(0\leqslant \varphi(t) \leqslant 1\) for all \(t\geqslant0\). Suppose that \(\varphi^{(3)} \in C^s(\R^n) \cap L^\infty(\R^n)\) is the solution to \begin{align}
\begin{PDE}
(-\Delta)^s \varphi^{(3)} &= 0 &\text{in } B_1, \\
\varphi^{(3)}(x) &= \varphi(x_1)  &\text{in } \R^n \setminus B_1.
\end{PDE} \label{pa7rDaw7}
\end{align}

Then \(\varphi^{(3)}\) is antisymmetric and there exists \(C>1\) depending only on \(n\) and \(s\) such that \begin{align*}
C^{-1}x_1  \leqslant \varphi^{(3)}(x) \leqslant Cx_1
\end{align*} for all \(x\in B_{1/2}^+\).
\end{lem}

\begin{proof}
Via the Poisson kernel representation, see \cite[Section 15]{MR3916700}, and using that \(\varphi\) is an odd function, we may write \begin{align*}
\varphi^{(3)}(x)  &= C \int_{\R^n \setminus B_1} \bigg (  \frac{1-\vert x \vert^2}{\vert y \vert^2-1}\bigg )^s \frac{\varphi(y_1)}{\vert x -y \vert^n} \dd y  \\
&= C \int_{\R^n_+ \setminus B_1^+} \bigg (  \frac{1-\vert x \vert^2}{\vert y \vert^2-1}\bigg )^s \bigg ( \frac 1 {\vert x -y \vert^n} - \frac 1 {\vert x_\ast -y \vert^n}\bigg )  \varphi(y_1)\dd y.
\end{align*} {F}rom this formula, we immediately obtain that \(\varphi^{(3)}\) is antisymmetric (this can also be argued by the uniqueness of solutions to~\eqref{pa7rDaw7}). Then, by an analogous computation to~\eqref{LxZU6} (just replacing~$n+2s$
with~$n$), \begin{align*}
\varphi^{(3)}(x) &\leqslant C  x_1 \int_{\R^n_+ \setminus B_1^+}  \frac{ y_1\varphi(y_1)}{(\vert y \vert^2-1)^s\vert x-y \vert^{n+2}}  \dd y \\
&\leqslant  C  x_1 \left[\int_{B_2^+ \setminus B_1^+}  \frac{ y_1}{(\vert y \vert^2-1)^s }  \dd y
+\int_{\R^n_+ \setminus B_2^+}  \frac{ y_1}{(\vert y \vert^2-1)^s(\vert y \vert - 1)^{n+2}}  \dd y\right] \\
&\leqslant Cx_1
\end{align*} for all \(x\in B_{1/2}^+\). Similarly, using now~\eqref{buKHzlE6} (replacing~$n+2s$
with~$n$), we have that
\begin{align*}
\varphi^{(3)}(x) &\geqslant  C x_1 \int_{\R^n_+ \setminus B_1^+}   \frac{ y_1\varphi(y_1)}{(\vert y \vert^2-1)^s\vert x_\ast -y \vert^{n+2}} \dd y \\
&\geqslant C x_1 \int_{\{y_1>2\}}   \frac{ y_1}{(\vert y \vert^2-1)^s(\vert y \vert +1)^{n+2}} \dd y \\
&\geqslant C x_1
\end{align*} for all \(x\in B_{1/2}^+\). 
\end{proof}

Now we can give the proof of \thref{EP5Elxbz}. 

\begin{proof}[Proof of \thref{EP5Elxbz}]
Dividing through by \(\Anorm{u}+M\), we can also assume that \((-\Delta)^s u + cu \leqslant x_1\) in \(B_\rho^+\) and \(\Anorm{u}\leqslant 1\).  Moreover, as explained at the start of the proof of \thref{g9foAd2c}, via \thref{tZVUcYJl} (ii) (after rescaling), it is not restrictive to assume \(u \in C^\infty_0(\R^n)\) and it is antisymmetric. 

Furthermore, we point out that the claim in \thref{EP5Elxbz} is obviously true if~$u\le0$
in~$B^+_\rho$, hence we suppose that~$\{u>0\}\cap B^+_\rho\ne\varnothing$.

Let \(\varphi^{(3)}\) be as in \thref{KaMmndyO} and let \(\zeta(x) := \varphi^{(3)}(x/(2\rho)) \). Suppose that \(\tau \geqslant 0\) is the smallest value such that \begin{align*}
u(x) &\leqslant \tau \zeta (x) (\rho - \vert x \vert )^{-n-2} \qquad \text{in } B_\rho^+. 
\end{align*}
The existence of such a \(\tau\) follows from a similar argument to the one presented in the footnote at the bottom of p. 11. Notice that~$\tau>0$.
To complete the proof we will show that \(\tau \leqslant C_\rho\) with \(C_\rho\) independent of \(u\).  Since \(u\) is continuously differentiable, two possibilities can occur: \begin{itemize}
\item[Case 1:] There exists \(a \in B_\rho^+\) such that \begin{align*}
u(a) &= \tau \zeta (a) (\rho - \vert a \vert )^{-n-2}.
\end{align*} 
\item[Case 2:] There exists \(a\in B_\rho \cap \{x_1=0\}\) such that \begin{align*}
\partial_1 u(a) = \tau \partial_1 \big \vert_{x=a} \big ( \zeta (x) (\rho - \vert x \vert )^{-n-2} \big )= \tau (\partial_1 \zeta (a)) (\rho - \vert a \vert )^{-n-2}  .
\end{align*} 
\end{itemize}
Let \(d:= \rho - \vert a\vert \) and define \(U\subset B_\rho^+\) as follows: if Case 1 occurs let \begin{align*}
U := \bigg \{ x \in B_\rho^+ \text{ s.t. } \frac{u(x)}{\zeta(x)} > \frac{u(a)}{2\zeta (a)} \bigg \};
\end{align*} otherwise if Case 2 occurs then let \begin{align*}
U:=\bigg \{ x \in B_\rho^+ \text{ s.t. }  \frac{u(x)}{\zeta(x)} > \frac{\partial_1u(a)}{2 \partial_1\zeta (a)} \bigg \}. 
\end{align*} Since  \(u(a) = \tau \zeta(a) d^{-n-2}\) in Case 1 and \(\partial_1u(a) = \tau \partial_1\zeta(a) d^{-n-2}\) in Case 2, we may write \begin{equation}\label{ei395v7b865998cn754mx984zUUU}
U =  \bigg \{ x \in B_\rho^+ \text{ s.t. } u(x)> \frac 1  2 \tau d^{-n-2} \zeta(x) \bigg \}
\end{equation}which is valid in both cases.

Then, we have that, for all \(r\in(0,d)\),
\begin{align*}
C_\rho &\geqslant \int_{B_\rho^+} y_1| u(y)| \dd y \geqslant \frac 1  2 \tau d^{-n-2} \int_{U \cap B_r(a)} y_1 \zeta(y) \dd y .
\end{align*} 
As a consequence, by \thref{KaMmndyO}, we have that \begin{align}
\int_{U \cap B_r(a)} y_1^2 \dd y 
\le C_\rho\int_{U \cap B_r(a)} y_1 \zeta(y) \dd y 
\leqslant \frac{ C_\rho d^{n+2} } \tau  \qquad \text{for all } r\in(0,d). \label{TTWpDkie}
\end{align} Next, we make the following claim. 

\begin{claim}
There exists~$\theta_0\in(0,1)$ depending only on \(n\), \(s\), \(\| c\|_{L^\infty(B_\rho(e_1))}\), and \(\rho\) such that
if~$\theta\in(0,\theta_0]$ there exists~\(C>0\)
depending only on \(n\), \(s\), \(\| c\|_{L^\infty(B_\rho(e_1))}\), \(\rho\), and~$\theta$ such that
\begin{itemize}
\item In Case 1: 
\begin{enumerate}[(i)]
\item If \(a_1 \geqslant   \theta d/16 \) then \begin{align*}
\big  \vert B_{(\theta d)/64}(a) \setminus U  \big \vert   &\leqslant \frac 1 4 \big  \vert B_{(\theta d)/64}  \big \vert  +\frac{C d^n} \tau .
\end{align*}
\item If \(a_1 <   \theta d/16 \) then \begin{align*}
\int_{B_{(\theta d)/64}^+(a) \setminus U } x_1^2 \dd x &\leqslant  \frac 1 4 \int_{B_{(\theta d)/64}^+(a) } x_1^2 \dd x + \frac{C d^{n+2}}\tau .
\end{align*}
\end{enumerate} 
\item In Case 2:  \begin{align*}
\int_{B_{(\theta d)/64}^+(a)\setminus U } x_1^2\dd x &\leqslant\frac 1 4 \int_{B_{(\theta d)/64}^+ (a)} x_1^2 \dd x + \frac{ C d^{n+2}}\tau.
\end{align*}
\end{itemize} In particular, neither~\(\theta\) nor~\(C\) depend on \(\tau\), \(u\), or \(a\).
\end{claim} 

We withhold the proof of the claim until the end. Assuming that
the claim is true, we complete the proof of \thref{EP5Elxbz} as follows.
 
 If Case 1(i) occurs then for all \(y\in B_{(\theta_0 d)/64} (a)\)
 we have that~\(y_1>a_1-( \theta_0 d)/64 \geqslant C d\), and so \begin{align*}
\int_{U \cap B_{(\theta_0 d)/64}(a)} y_1^2 \dd y \geqslant C d^2 \cdot \big \vert U \cap B_{(\theta_0 d)/64}(a) \big \vert .
\end{align*} Hence, from~\eqref{TTWpDkie} (used here with~$r:=(\theta_0d)/64$), we have that \begin{align*}
\vert U \cap B_{( \theta_0 d)/64}(a) \big \vert \leqslant \frac{C_\rho d^n}{\tau}. 
\end{align*}
Then using the claim, we find that \begin{align*}
\frac{C_\rho d^n}{\tau} \geqslant \big  \vert B_{(\theta_0 d)/64} \big \vert  - \big  \vert B_{(\theta_0 d)/64}(a) \setminus U  \big \vert 
\geqslant  \frac 3 4 \big  \vert B_{(\theta_0 d)/64}  \big \vert  -\frac{C d^n} \tau 
\end{align*}which gives that \(\tau \leqslant C_\rho\) in this case.

If Case 1(ii) or Case 2 occurs then from~\eqref{TTWpDkie} (used here with~$r:=(\theta_0d)/64$)
and the claim, we have that \begin{equation}\label{dewiotbv5748976w4598ty}
\frac{C_\rho d^{n+2}}{\tau} \geqslant \int_{B_{(\theta_0 d)/64}^+(a) } x_1^2 \dd x - \int_{B_{( \theta_0 d)/64}^+(a) \setminus U } x_1^2 \dd x 
\geqslant \frac 3 4 \int_{B_{( \theta_0 d)/64}^+(a) } x_1^2 \dd x - \frac{C d^{n+2}}\tau .
\end{equation}
We now observe that, given~$r\in(0,d)$, if~$x\in B_{r/4}\left(a+\frac34 re_1\right)\subset B^+_{r}(a)$
then~$x_1\ge a_1+\frac34 r-\frac{r}4\ge\frac{r}2$,
and thus
\begin{equation}\label{sdwet68980poiuytrkjhgfdmnbvcx23456789}
\int_{B_{r}^+(a) } x_1^2 \dd x\ge
\int_{B_{r/4}(a+(3r)/4 e_1) } x_1^2 \dd x\ge \frac{r^2}4\,|B_{r/4}(a+(3r)/4 e_1) |=C r^{n+2},\end{equation}
for some~$C>0$ depending on~$n$. Exploiting this formula with~$r:=(\theta_0 d)/64$ into~\eqref{dewiotbv5748976w4598ty},
we obtain that
$$ \frac{C_\rho d^{n+2}}{\tau} \geqslant C \left(\frac{ \theta_0 d}{64}\right)^{n+2 }- \frac{C d^{n+2}}\tau,$$
which gives that \(\tau \leqslant C_\rho\) as required.

Hence, we now focus on the proof of the claim. 
Let \(\theta\in(0,1)\) be a small constant to be chosen later. By translating, we may also assume without loss of generality that \(a'=0\) (recall that~\(a'=(a_2,\dots,a_n)\in \R^{n-1}\)).

For each \(x\in B_{\theta d/2}^+(a)\), we have that~\(\vert x \vert \leqslant \vert a \vert +\theta d/2=\rho- (1-\theta/2)d\). Hence, in both Case 1 and Case 2, 
\begin{align}
u(x) \leqslant \tau d^{-n-2}   \bigg (1- \frac \theta 2 \bigg )^{-n-2}\zeta (x) \qquad \text{in }B_{\theta d/2}^+(a). \label{ln9LcJTh}
\end{align} Let \begin{align*}
v(x):=  \tau d^{-n-2}   \bigg (1- \frac \theta 2 \bigg )^{-n-2}\zeta (x)-u(x)  \qquad \text{for all } x\in \R^n. 
\end{align*} We have that \(v\) is antisymmetric and \(v\geqslant 0\) in \(B_{\theta d/2}^+(a) \) due to~\eqref{ln9LcJTh}. Moreover, since \(\zeta\) is \(s\)-harmonic in \(B_\rho^+ \supset B_{\theta d/2}^+(a)\), for all \(x\in B_{\theta d/2}^+(a)\), \begin{align*}
(-\Delta)^s v(x) +c(x)v(x) &= -(-\Delta)^s u(x) - c(x) u(x) +c(x) \tau d^{-n-2} \bigg (1- \frac \theta 2 \bigg )^{-n-2}\zeta (x)\\
&\geqslant -x_1 -C  \tau d^{-n-2} \|c^-\|_{L^\infty(B_\rho^+)} \bigg (1- \frac \theta 2 \bigg )^{-n-2}\zeta (x).
\end{align*} Taking \(\theta\) sufficiently small and using that \(\zeta (x)\leqslant C x_1\) (in light of~\thref{KaMmndyO}),
we obtain \begin{align}
(-\Delta)^s v(x) +c(x)v(x)&\geqslant -C \big ( 1 + \tau d^{-n-2}  \big )x_1  \qquad \text{in } B_{\theta d/2}^+(a). \label{ulzzcUwf}
\end{align}

Next, we define \(w(x):= v^+(x)\) for all \(x\in \R^n_+\) and \(w(x):= -w(x_\ast)\) for all \(x\in \overline{\R^n_-}\). 
We point out that, in light of~\eqref{ln9LcJTh}, $w$ is as regular as~$v$ in~$B_{\theta d/2}^+(a)$, and thus
we can compute the fractional Laplacian of~$w$ in~$B_{\theta d/2}^+(a)$ in a pointwise sense.

We also observe that \begin{align*}
(w-v)(x) (x) &=  \begin{cases}
0 &\text{if } x\in \R^n_+ \cap \{ v\geqslant 0\} ,\\
u(x)-\tau d^{-n-2} \big (1- \frac \theta 2 \big )^{-n-2}\zeta (x),&\text{if } x\in \R^n_+ \cap \{ v< 0\}.
\end{cases}
\end{align*} In particular, \(w-v\leqslant |u|\) in \(\R^n_+\).
Thus, for all \(x\in B_{\theta d/2}^+(a)\), \begin{align*}
(-\Delta)^s (w-v)(x)&\geqslant -C \int_{\R^n_+\setminus B_{\theta d/2}^+(a)} \left(\frac 1 {\vert x -y \vert^{n+2s}} - \frac 1 {\vert x_\ast -y \vert^{n+2s}} \right) |u(y)|\dd y .
\end{align*}  Moreover, by \thref{ltKO2}, for all~$x \in B_{\theta d/4}^+(a)$,
\begin{equation}
(-\Delta)^s (w-v)(x) \geqslant -C(\theta d)^{-n-2s-2} \Anorm{u}
x_1 \geqslant -C(\theta d)^{-n-2s-2}x_1 . \label{y8tE2pf9}
\end{equation} Hence, by~\eqref{ulzzcUwf} and~\eqref{y8tE2pf9}, we obtain  \begin{align}
(-\Delta)^s w +cw &=(-\Delta)^sv +cv +(-\Delta)^s(w-v) \nonumber \\
&\geqslant -C \bigg ( 1 + \tau d^{-n-2}  +(\theta d)^{-n-2s-2}  \bigg )x_1 \nonumber \\
&\geqslant -C \bigg (  (\theta d)^{-n-2s-2} + \tau d^{-n-2}  \bigg )x_1\label{ayFE7nQK}
\end{align} in \( B_{\theta d/4}^+(a)\).

Next, let us consider Case 1 and Case 2 separately.  

\emph{Case 1:} Suppose that \(a\in B_\rho^+\) and let \(\tilde w (x) = w (a_1 x)\) and \(\tilde c (x) = a_1^{2s} c (a_1 x)\). Then from~\eqref{ayFE7nQK}, we have that \begin{align}
(-\Delta)^s \tilde w(x) +\tilde c(x) \tilde w(x) &\geqslant - Ca_1^{2s+1} \bigg (  (\theta d)^{-n-2s-2} + \tau d^{-n-2}  \bigg )x_1 \label{zrghHZhX}
\end{align} for all \( x\in B_{\theta d/(4a_1)}^+(e_1)\). 

As in the proof of \thref{guDQ7}, we wish to apply the rescaled version of the weak Harnack inequality to \(\tilde w\); however, we cannot immediately apply either \thref{YLj1r} or \thref{g9foAd2c} to~\eqref{zrghHZhX}. To resolve this, let us split into a further two cases: (i) \(a_1 \geqslant   \theta d/16 \) and (ii) \(a_1<  \theta d/16\).

\emph{Case 1(i):} If \(a_1 \geqslant   \theta d/16 \) then \( B_{\theta d/(32a_1)}(e_1) \subset B_{\theta d/(4a_1)}^+(e_1)\) and for each \(x\in B_{\theta d/(32a_1)}(e_1)\) we have that~\(x_1<1+\theta d/(32a_1)\le1+1/4=5/4\). Therefore, from~\eqref{zrghHZhX}, we have that \begin{align*}
(-\Delta)^s \tilde w(x) +\tilde c(x) \tilde w(x) &\geqslant - Ca_1^{2s+1}  \big (  (\theta d)^{-n-2s-2} + \tau d^{-n-2}  \big ) 
\end{align*} for all \( x\in B_{\theta d/(32a_1)}(e_1)\).

On one hand, by~\thref{YLj1r} (used here with~$\rho:=\theta d/(32a_1)$), \begin{align*}&
\left( \frac{\theta d}{64} \right)^{-n} \int_{B_{\theta d/64}(a)}  w (x) \dd x \\= &\; \left( \frac{\theta d}{32a_1} \right)^{-n}  \int_{B_{\theta d/(64a_1)}(e_1)} \tilde w (x) \dd x \\
\leqslant&\; C \left(  \tilde w(e_1) +  a_1 (\theta d)^{-n-2} + \tau a_1 \theta^{2s} d^{-n+2s-2} \right) \\
\leqslant& \;C  \tau d^{-n-2} \bigg (  \bigg (1- \frac \theta 2 \bigg )^{-n-2}-1 \bigg ) a_1 +Ca_1(\theta d)^{-n-2} + C\tau a_1 \theta^{2s} d^{-n+2s-2}
\end{align*} using also \thref{KaMmndyO} and that \(u(a) = \tau d^{-n-2} \zeta(a)\). 

On the other hand, by the definition of~$U$ in~\eqref{ei395v7b865998cn754mx984zUUU},
\begin{align}
B_r(a) \setminus U  &\subset \left\{ \frac{w}{\zeta} > \tau d^{-n-2} \left(  \left(1 - \frac \theta 2 \right)^{-n-2} -\frac 1 2 \right)  \right\} \cap B_r(a), \qquad \text{ for all } r\in\left(0,\frac12\theta d\right),\label{I9FvOxCz}
\end{align} and so \begin{align*}
(\theta d)^{-n} \int_{B_{\theta d/64}(a)}  w(x) \dd x &\geqslant \tau \theta^{-n} d^{-2n-2} \bigg ( \bigg (1 - \frac \theta 2 \bigg )^{-n-2} -\frac 1 2 \bigg ) \int_{B_{\theta d/64}(a) \setminus U }  \zeta (x) \dd x \\
&\geqslant  C \tau \theta^{-n} d^{-2n-2} \int_{B_{\theta d/64}(a) \setminus U } \zeta (x) \dd x
\end{align*} for \(\theta\) sufficiently small. Moreover, using that \(x_1>a_1-\theta d/64>Ca_1\) and \thref{KaMmndyO}, we have that\begin{align*}
\int_{B_{\theta d/64}(a) \setminus U } \zeta (x) \dd x \geqslant C\int_{B_{\theta d/64}(a) \setminus U } x_1 \dd x \geqslant C
a_1\cdot \big  \vert B_{\theta d/64}(a) \setminus U  \big \vert . 
\end{align*}
Thus,\begin{align*}
\big  \vert B_{\theta d/64}(a) \setminus U  \big \vert   &\leqslant C  (\theta d)^n \bigg (  \bigg (1- \frac \theta 2 \bigg )^{-n-2}-1 \bigg )  +\frac{C\theta^{-2} d^n} \tau + C (\theta d)^{n+2s} \\
&\leqslant  C  (\theta d)^n \bigg (  \bigg (1- \frac \theta 2 \bigg )^{-n-2}-1 +\theta^{2s} \bigg )  +\frac{C\theta^{-2} d^n} \tau ,
\end{align*} using also that \(d^{n+2s}<d^n\). 

Finally, we can take \(\theta\) sufficiently small so that \begin{align*}
C (\theta d)^n \bigg (  \bigg (1- \frac \theta 2 \bigg )^{-n-2}-1 +\theta^{2s} \bigg ) \leqslant \frac 1 4 \big \vert B_{\theta d/64} \big \vert 
\end{align*}
which gives
$$ \big  \vert B_{\theta d/64}(a) \setminus U  \big \vert\le \frac14 \big \vert B_{\theta d/64} \big \vert
 +\frac{Cd^n} \tau.
$$
This concludes the proof of the claim in Case 1(i). 

\emph{Case 1(ii):} Let \(a_1<  \theta d/16\) and fix \(R:= \frac 1 {2} \big ( \sqrt{ ( \theta d/(4a_1))^2-1} +2 \big )\). Observe that \begin{align*}
2 < R< \sqrt{ \left(\frac{ \theta d}{4a_1}\right)^2-1}.
\end{align*} Hence, \(e_1 \in B_{R/2}^+\). Moreover, if \(x\in B_R^+\) then \begin{align*}
\vert x -e_1\vert^2 <1+R^2< ( \theta d/(4a_1))^2,
\end{align*} so \(B_R^+ \subset B_{\theta d/(4a_1)}^+(e_1)\). 

Thus, applying \thref{g9foAd2c} to the equation in~\eqref{zrghHZhX} in \(B_R^+\), we obtain \begin{align*}&
a_1^{-n-1}\bigg ( \frac R 2\bigg )^{-n-2} \int_{B_{a_1R/2}^+} x_1  w(x) \dd x \\
=&\;  \bigg ( \frac R 2\bigg )^{-n-2} \int_{B_{R/2}^+} x_1 \tilde w(x) \dd x \\
\leqslant&\; C \left( \inf_{B_{R/2}^+} \frac{\tilde w(x)}{x_1} +  a_1^{2s+1}R^{2s}  (\theta d)^{-n-2s-2} + \tau d^{-n-2}a_1^{2s+1}R^{2s}   \right) \\
\leqslant\;& C \tau d^{-n-2}   \left( \left(1- \frac \theta 2 \right)^{-n-2}-1 \right)\zeta(a)
+ C  a_1^{2s+1}R^{2s}  (\theta d)^{-n-2s-2} +C  \tau d^{-n-2}a_1^{2s+1}R^{2s}  
\end{align*} using that \(e_1 \in B_{R/2}^+\).

Since \(R \leqslant C \theta d/a_1\) and \(\zeta(a)\leqslant Ca_1\) by \thref{KaMmndyO}, it follows that
\begin{eqnarray*}&&
a_1^{-1} \bigg ( \frac R {2a_1} \bigg )^{-n} \int_{B_{a_1R/2}^+} x_1  w(x) \dd x \\
&&\qquad\leqslant C \tau d^{-n-2}   \bigg ( \bigg (1- \frac \theta 2 \bigg )^{-n-2}-1 \bigg )a_1  + C  a_1  (\theta d)^{-n-2} +C  \tau a_1 \theta^{2s} d^{-n+2s-2} .
\end{eqnarray*}

On the other hand, we claim that
\begin{equation}\label{skweogtry76t678tutoy4554yb76i78io896}
B_{(\theta d)/64}^+(a) \setminus U \subset B_{a_1R/2}^+ .\end{equation} Indeed,
if~$x\in B_{(\theta d)/64}^+ \setminus U $ then
$$ |x|\le |x-a|+|a|\le \frac{\theta d}{64} +a_1\le a_1\left(\frac{\theta d}{64a_1}+1\right).
$$
Furthermore,
\begin{eqnarray*}
&&R\ge   \frac 1 {2} \sqrt{ \left(\frac{\theta d}{4a_1}\right)^2-\left(\frac{\theta d}{16a_1}\right)^2} +1\ge
\frac{\sqrt{15}\theta d}{32a_1}+1\ge\frac{3\theta d}{32a_1}+1\\&&\qquad
=2\left(\frac{3\theta d}{64a_1}+\frac12\right)=2\left(\frac{\theta d}{64a_1}+\frac{\theta d}{32a_1}+\frac12\right)\ge
2\left(\frac{\theta d}{64a_1}+1\right).
\end{eqnarray*}
{F}rom these observations we obtain that if~$x\in B_{(\theta d)/64}^+ \setminus U $ then~$|x|\le a_1R/2$,
which proves~\eqref{skweogtry76t678tutoy4554yb76i78io896}.

Hence, by~\eqref{I9FvOxCz} (used with~$r:=\theta d/64$), \eqref{skweogtry76t678tutoy4554yb76i78io896}, and \thref{KaMmndyO}, we have that \begin{align*}
a_1^{-1} \bigg ( \frac R {2a_1} \bigg )^{-n} \int_{B_{a_1R/2}^+} x_1  w(x) \dd x  &\geqslant a_1^{-n-1} R^{-n-2} \tau d^{-n-2} \bigg (  \bigg (1 - \frac \theta 2 \bigg )^{-n-2} -\frac 1 2 \bigg ) \int_{B_{ (\theta d)/64}^+(a) \setminus U} x_1  \zeta(x) \dd x \\
&\geqslant C \tau a_1 \theta^{-n-2}  d^{-2n-4} \int_{B_{ (\theta d)/64}^+(a) \setminus U } x_1^2 \dd x
\end{align*} for \(\theta\) sufficiently small. Thus, \begin{align*}
\int_{B_{( \theta d)/64}^+(a) \setminus U } x_1^2 \dd x &\leqslant C   (\theta d)^{n+2}   \bigg ( \bigg (1- \frac \theta 2 \bigg )^{-n-2}-1 \bigg ) + \frac{C d^{n+2}}\tau  +C( \theta  d)^{n+2s+2}\\
&\leqslant  C   (\theta d)^{n+2}   \bigg ( \bigg (1- \frac \theta 2 \bigg )^{-n-2}-1 +\theta^{2s} \bigg ) + \frac{C d^{n+2}}\tau  
\end{align*} using that \(d^{n+2s+2}<d^{n+2}\). 
Recalling formula~\eqref{sdwet68980poiuytrkjhgfdmnbvcx23456789}
and taking~\(\theta\) sufficiently small, we obtain that \begin{align*}
C(\theta d)^{n+2}   \bigg ( \bigg (1- \frac \theta 2 \bigg )^{-n-2}-1 +\theta^{2s} \bigg ) \leqslant \frac 1 4 \int_{B_{( \theta d)/64}^+ (a)} x_1^2 \dd x .
\end{align*} which concludes the proof in Case 1(b). 

\emph{Case 2:} In this case, we can directly apply \thref{g9foAd2c} to~\eqref{ayFE7nQK}. When we do this we find that \begin{align*}&
\bigg ( \frac{\theta d }{64} \bigg )^{-n-2} \int_{B_{\theta d/64}^+(a)} x_1  w (x)\dd x \\
\leqslant&\, C \bigg ( \partial_1  w (0)  +  (r\theta)^{-n-2} + \tau \theta^{2s} d^{-n+2s-2} \bigg )\\
=&\, C\tau d^{-n-2}  \bigg (   \bigg (1- \frac \theta 2 \bigg )^{-n-2}-1 \bigg )\partial_1\zeta(0) + C   (r\theta)^{-n-2} + C\tau \theta^{2s} d^{-n+2s-2}.
\end{align*} On the other hand,~\eqref{I9FvOxCz} is still valid in Case 2 so \begin{align*}
(\theta d )^{-n-2} \int_{B_{\theta d/64}^+(a)} x_1  w (x)\dd x &\geqslant  \tau  \theta^{-n-2} d^{-2n-4} \bigg (  \bigg (1 - \frac \theta 2 \bigg )^{-n-2} -\frac 1 2 \bigg ) \int_{B_{\theta d/64}^+(a)\setminus U } x_1 \zeta (x) \dd y \\
&\geqslant C \tau  \theta^{-n-2} d^{-2n-4} \int_{B_{\theta d/64}^+(a)\setminus U } x_1^2\dd x
\end{align*} using \thref{KaMmndyO} and taking \(\theta\) sufficiently small. Thus,  \begin{align*}
\int_{B_{\theta d/64}^+(a)\setminus U } x_1^2\dd x &\leqslant C (\theta d)^{n+2}  \bigg (   \bigg (1- \frac \theta 2 \bigg )^{-n-2}-1 \bigg ) + \frac{ C d^{n+2}}\tau + C (\theta d)^{n+2s+2}\\
&\leqslant  C (\theta d)^{n+2}  \bigg (   \bigg (1- \frac \theta 2 \bigg )^{-n-2}-1 +\theta^{2s} \bigg ) + \frac{ C d^{n+2}}\tau 
\end{align*} using that \(d^{n+2s+2}<d^{n+2}\). Then, by~\eqref{sdwet68980poiuytrkjhgfdmnbvcx23456789}, we may choose \(\theta\) sufficiently small so that \begin{align*}
 C (\theta d)^{n+2}  \bigg (   \bigg (1- \frac \theta 2 \bigg )^{-n-2}-1 +\theta^{2s} \bigg ) \leqslant \frac 1 4 \int_{B_{\theta d/64}^+(a)} x_1^2\dd x
\end{align*} which concludes the proof in Case 2. 

The proof of \thref{EP5Elxbz}
is thereby complete.
\end{proof}

\section{Proof of \protect\thref{Hvmln}} \label{lXIUl}

In this short section, we give the proof of \thref{Hvmln}. This follows from Theorems~\ref{DYcYH} and~\ref{C35ZH} along with a standard covering argument which we include here for completeness. 

\begin{proof}[Proof of \thref{Hvmln}] 
Recall that~\(\Omega^+=\Omega \cap \R^n_+\) and let \(B_{2R}(y)\) be a ball such that \(B_{2R}(y)\Subset\Omega^+\). We will first prove that there exists a constant \(C=C(n,s,R,y)\) such that \begin{align}
\sup_{B_R(y)} \frac{u(x)}{x_1} &\leqslant C \inf_{B_R(y)} \frac{u(x)}{x_1} .  \label{c4oOr}
\end{align}Indeed, if \(\tilde u (x) := u (y_1x+(0,y'))\) and~\(\tilde c (x):= y_1^{2s} c(y_1x+(0,y'))\) then  \begin{align*}
(-\Delta)^s \tilde u (x)+\tilde c \tilde u = 0, \qquad \text{in } B_{2R/y_1} (e_1). 
\end{align*} By \thref{DYcYH}, \begin{align*}
\sup_{B_R(y)} \frac{u(x)}{x_1}  \leqslant C \sup_{B_R(y)} u = C\sup_{B_{R/y_1}(e_1)} \tilde u  & \leqslant C \inf_{B_{R/y_1}(e_1)} \tilde u =C\inf_{B_R(y)} u \leqslant C\inf_{B_R(y)} \frac{u(x)}{x_1}
\end{align*} using that \(B_R(y) \Subset\R^n_+\).

Next, let \(\{a^{(k)}\}_{k=1}^\infty,\{b^{(k)}\}_{k=1}^\infty \subset \tilde \Omega^+\) be such that \begin{align*}
\frac{u(a^{(k)})}{a^{(k)}_1} \to \sup_{\Omega'}\frac{u(x)}{x_1} \quad \text{and} \quad \frac{u(b^{(k)})}{b^{(k)}_1} \to \inf_{\Omega'}\frac{u(x)}{x_1}
\end{align*} as \(k\to \infty\). After possibly passing to a subsequence, there exist~\(a\), \(b\in \overline{\tilde \Omega^+}\) such that~\(a^{(k)}\to a\) and~\(b^{(k)}\to b\). Let \(\gamma \subset \tilde \Omega^+\) be a curve connecting \(a\) and \(b\). By the Heine-Borel Theorem, there exists a finite collection of balls \(\{B^{(k)}\}_{k=1}^M\) with centres in \(\tilde \Omega \cap \{x_1=0\}\) such that \begin{align*}
\tilde \Omega \cap \{x_1=0\} \subset \bigcup_{k=1}^M B^{(k)}  \Subset\Omega . 
\end{align*} Moreover, if \( \tilde \gamma := \gamma \setminus \bigcup_{k=1}^M B^{(k)}\) then there exists a further collection of balls \(\{B^{(k)}\}_{k=M+1}^N\) with centres in \(\tilde \gamma\) and radii equal to \( \frac 12 \dist(\tilde \gamma, \partial( \Omega^+))\) such that \begin{align*}
\tilde \gamma \subset \bigcup_{k=M+1}^N B^{(k)}  \Subset \Omega^+.
\end{align*} By construction, \(\gamma\) is covered by \(\{B^{(k)}\}_{k=1}^n\). Thus, iteratively applying \thref{C35ZH} (after translating and rescaling) to each \(B^{(k)}\), \(k=1,\dots, M\), and~\eqref{c4oOr} to each \(B^{(k)}\), \(k=M+1,\dots, N\), we obtain the result. 
\end{proof} 

\appendix 
 
\section{A counterexample} \label{j4njb}
In this appendix, we demonstrate that \thref{Hvmln} is in general false if we do not assume antisymmetry. We will do this by constructing a sequence of functions \(\{ u_k\}_{k=1}^\infty \subset C^\infty(B_1(2e_1)) \cap L^\infty(\R^n)\) such that \((-\Delta)^su_k=0\) in \(B_1(2e_1)\), \(u_k \geqslant 0\) in \(\R^n_+\), and \begin{align*}
\frac{\sup_{B_{1/2}(2e_1)}u_k}{\inf_{B_{1/2}(2e_1)}u_k} \to + \infty  \qquad \text{as } k\to \infty.
\end{align*} The proof will rely on the mean value property of \(s\)-harmonic functions. 

Suppose that \(M\geqslant 0\) and \(\zeta_1\), \(\zeta_2\) are smooth functions such that \(0\leqslant \zeta_1,\zeta_2\leqslant 1\) in \(\R^n \setminus B_1(2e_1)\), and \begin{align*}
\zeta_1(x) = 
\begin{cases}
0 &\text{in } \R^n_+ \setminus B_1(2e_1),\\
1 &\text{in } \R^n_- \setminus \{x_1>-1\}
\end{cases} \qquad{\mbox{and}}\qquad \zeta_2(x) =\begin{cases}
0 &\text{in } \R^n_+ \setminus B_1(2e_1) ,\\
1 &\text{in } B_{1/2}(-2e_1).
\end{cases}
\end{align*}Then let \(v\) and \(w_M\) be the solutions to \begin{align*}
\begin{PDE}
(-\Delta)^s v &= 0 &\text{in }B_1(2e_1), \\
v&=\zeta_1 &\text{in } \R^n\setminus  B_1(2e_1)
\end{PDE}\qquad {\mbox{and}}\qquad \begin{PDE}
(-\Delta)^s w_M &= 0 &\text{in }B_1(2e_1), \\
w_M&=-M\zeta_2 &\text{in } \R^n\setminus  B_1(2e_1),
\end{PDE}
\end{align*} respectively. Both \(v\) and \(w_M\) are in \( C^\infty(B_1(2e_1)) \cap L^\infty(\R^n)\) owing to standard regularity theory. We want to emphasise that \(w_M\) depends on the parameter \(M\) (as indicated by the subscript) but \(v\) does not. Define \(\tilde u_M:= v+w_M\). Since \(w_M \equiv 0 \) when \(M=0\) and, by the strong maximum principle, \(v>0\) in \(B_1(2e_1)\), we have that \begin{align*}
\tilde u_0 >0 \qquad \text{in } B_1 (2e_1).
\end{align*} Hence, we can define \begin{align*}
\bar M := \sup \{ M \geqslant 0 \text{ s.t. } \tilde u_M >0 \text{ in } B_1 (2e_1)\}
\end{align*} though \(\bar M\) may possibly be infinity. We will show that \(\bar M\) is in fact finite and that \begin{align}
\inf_{B_1(2e_1)} \tilde u_{\bar M} = 0. \label{fDMti}
\end{align} Once these two facts have been established, we complete the proof as follows: set \(u_k := \tilde u_{\bar M - 1/k}\). By construction, \(u_k \geqslant 0\) in \(\R^n_+\) and \((-\Delta)^su_k=0\) in \(B_1(2e_1)\). Moreover, by the maximum principle, \(v\leqslant 1\) and \(w_M \leqslant 0\) in \(\R^n\), so \( u_k \leqslant 1\) in \(\R^n\). Hence, \begin{align*}
\frac{\sup_{B_{1/2}(2e_1)}u_k}{\inf_{B_{1/2}(2e_1)}u_k} &\leqslant \frac 1 {\inf_{B_1(2e_1)}u_k} \to + \infty
\end{align*} as \(k \to \infty\). 

Let us show that \(\bar M\) is finite. Since \(w_M\) is harmonic in \(B_1(2e_1)\), the mean value property for \(s\) harmonic functions, for example see \cite[Section 15]{MR3916700}, tells us that\begin{align*}
w_M(2e_1) &= -CM \int_{\R^n \setminus B_1(2e_1)} \frac{\zeta_2(y)}{(\vert y \vert^2 -1 )^s \vert y \vert^n} \dd y .
\end{align*} Then, since \(\zeta_2\equiv 1\) in \(B_{1/2}(-2e_1)\) and \(\zeta_2\geqslant0\), \begin{align*}
\int_{\R^n \setminus B_1(2e_1)} \frac{\zeta_2(y)}{(\vert y \vert^2 -1 )^s \vert y \vert^n} \dd y &\geqslant \int_{B_{1/2}(-2e_1)} \frac{\dd y}{(\vert y \vert^2 -1 )^s \vert y \vert^n}\geqslant C.  
\end{align*} It follows that \(w_M(2e_1) \leqslant -CM\) whence \begin{align*}
\tilde u_M (2e_1) \leqslant 1-CM \leqslant 0
\end{align*} for \(M\) sufficiently large. This gives that \(\bar M\) is finite.

Now we will show \eqref{fDMti}. For the sake of contradiction, suppose that there exists \(a>0\) such that \(\tilde u_{\bar M } \geqslant a \) in \(B_1(2e_1)\). Suppose that \(\varepsilon>0\) is small and let \(h^{(\varepsilon)} := \tilde u_{\bar M+\varepsilon}-\tilde u_{\bar M} = w_{\bar M+\varepsilon}-w_{\bar M}\). Then \(h^{(\varepsilon)}\) is \(s\)-harmonic in \(B_1(2e_1)\) and \begin{align*}
h^{(\varepsilon)}(x) = - \varepsilon \zeta_2(x) \geqslant -\varepsilon \qquad \text{in } \R^n \setminus B_1(2e_1).
\end{align*}Thus, using the maximum principle we conclude that \(h^{(\varepsilon)} \geqslant -\varepsilon\) in \(B_1(2e_1)\),
which in turn gives that \begin{align*}
\tilde u_{\bar M+\varepsilon} =\tilde u +h^{(\varepsilon)}  \geqslant a-\varepsilon >0 \qquad \text{in } B_1(2e_1)
\end{align*} for \(\varepsilon\) sufficiently small. This contradicts the definition of \(\bar M\). 

\section{Alternate proof of \protect\thref{C35ZH} when \(c\equiv0\)} \label{4CEly}

In this appendix, we will provide an alternate elementary proof of \thref{C35ZH} in the particular case \(c\equiv0\) and \(u\in \mathscr L_s(\R^n)\). More precisely, we prove the following.

\begin{thm} \thlabel{oCuv7Zs2}
Let \(u\in C^{2s+\alpha}(B_1) \cap \mathscr L_s(\R^n) \) for some \(\alpha>0\) with \(2s+\alpha\) not an integer. Suppose that \(u\) is antisymmetric, non-negative in \(\R^n_+\), and \(s\)-harmonic in \(B_1^+\).

Then there exists \(C>0\) depending only on \(n\) and \(s\) such that \begin{align*}
\sup_{B_{1/2}^+} \frac{u(x)}{x_1 } &\leqslant C \inf_{B_{1/2}^+} \frac{u(x)}{x_1 }. 
\end{align*} Moreover, \(\inf_{B_{1/2}^+} \frac{u(x)}{x_1 }\) and \(\sup_{B_{1/2}^+} \frac{u(x)}{x_1 }\) are comparable to \(\Anorm{u}\). 
\end{thm} 

Except for the statement  \(\inf_{B_{1/2}^+} \frac{u(x)}{x_1 }\) and \(\sup_{B_{1/2}^+} \frac{u(x)}{x_1 }\) are comparable to \(\Anorm{u}\), Theorem~\ref{oCuv7Zs2} was proven in \cite{ciraolo2021symmetry}. Both the proof presented here and the proof in \cite{ciraolo2021symmetry} rely on the Poisson kernel representation for \(s\)-harmonic functions in a ball. Despite this, our proof is entirely different to the proof in \cite{ciraolo2021symmetry}. This was necessary to prove that~\(\inf_{B_{1/2}^+} \frac{u(x)}{x_1 }\) and \(\sup_{B_{1/2}^+} \frac{u(x)}{x_1 }\) are comparable to \(\Anorm{u}\) which does not readily follow from the proof in \cite{ciraolo2021symmetry}. Our proof of Theorem~\ref{oCuv7Zs2} is a consequence of a new mean-value formula for antisymmetric \(s\)-harmonic functions (Proposition~\ref{cqGgE}) which we believe to be interesting in and of itself.

We first prove an alternate expression of the Poisson kernel representation formula for antisymmetric functions.

\begin{lem} \thlabel{Gku6y}
Let \(u\in C^{2s+\alpha}(B_1) \cap \mathscr L_s (\R^n)\) and \(r\in(0, 1] \).  If \(u\) is antisymmetric and \((-\Delta)^s u = 0 \) in~\(B_1\).

Then \begin{align}
u(x) & =  \gamma_{n,s} \int_{\R^n_+ \setminus B_r^+} \bigg ( \frac{r^2- \vert x \vert^2 }{\vert  y \vert^2 -r^2 } \bigg )^s \bigg ( \frac 1 {\vert x - y \vert^n }-\frac 1 {\vert x_\ast - y \vert^n }  \bigg ) u(y) \dd y \label{BPyAO}
\end{align} for all \(x\in B_1^+\).  Here~$\gamma_{n,s}$ is given in~\eqref{sry6yagamma098765}.
\end{lem}

We remark that
since \(y \mapsto  \frac 1 {\vert x - y \vert^n }-\frac 1 {\vert x_\ast - y \vert^n }  \) is antisymmetric for each \(x\in \R^n_+\), we can rewrite~\eqref{BPyAO} as \begin{align} \label{ZrGnZFNg}
u(x) & = \frac 1 2  \gamma_{n,s} \int_{\R^n \setminus B_r} \bigg ( \frac{r^2- \vert x \vert^2 }{\vert  y \vert^2 -r^2 } \bigg )^s \bigg ( \frac 1 {\vert x - y \vert^n }-\frac 1 {\vert x_\ast - y \vert^n }  \bigg ) u(y) \dd y.
\end{align}

\begin{proof}[Proof of Lemma~\ref{Gku6y}]
The Poisson representation formula \cite{aintegrales} gives\begin{align}
u(x) & =\gamma_{n,s} \int_{\R^n \setminus B_r} \bigg ( \frac{r^2- \vert x \vert^2 }{\vert  y \vert^2 -r^2 } \bigg )^s \frac {u(y)} {\vert x - y \vert^n }\dd y \label{vDlwi}
\end{align} where \(\gamma_{n,s}= \frac{\sin(\pi s)\Gamma (n/2)}{\pi^{\frac n 2 +1 }}\); see also \cite[p.112,p.122]{MR0350027} and \cite[Section 15]{MR3916700}. Splitting the integral in \eqref{vDlwi} into two separate integrals over \(\R^n_+ \setminus B_r^+\) and \(\R^n_- \setminus B_r^-\) respectively, then making the change of variables \(y \to y_\ast\) in the integral over \(\R^n_- \setminus B_r^-\) and using that \(u\) is antisymmetric, we obtain \begin{align*}
u(x) & =\gamma_{n,s} \int_{\R^n_+ \setminus B_r^+} \bigg ( \frac{r^2- \vert x \vert^2 }{\vert  y \vert^2 -r^2 } \bigg )^s \bigg ( \frac 1 {\vert x - y \vert^n }-\frac 1 {\vert x_\ast - y \vert^n }  \bigg ) u(y) \dd y .  \qedhere
\end{align*} 
\end{proof}

Now that we have proven the Poisson kernel formula for antisymmetric functions in Lemma~\ref{Gku6y},
we now establish Proposition~\ref{cqGgE}.

\begin{proof}[Proof of Proposition~\ref{cqGgE}]
Let \(h\in(0,1)\). It follows from Lemma~\ref{Gku6y} that\begin{align}
\frac{u(he_1)} h  & = \frac 1 {h} \gamma_{n,s} \int_{\R^n_+ \setminus B_r^+} \bigg ( \frac{r^2- h^2 }{\vert  y \vert^2 -r^2 } \bigg )^s \bigg ( \frac 1 {\vert he_1 - y \vert^n }-\frac 1 {\vert he_1+ y \vert^n }  \bigg ) u(y) \dd y  \label{CNZLU}
\end{align} for all \(r>h\). Taking a Taylor series in \(h\) about \(0\), we have the pointwise limit \begin{align*}
\lim_{h \to 0 } \frac 1 h \bigg ( \frac 1 {\vert he_1 - y \vert^n }-\frac 1 {\vert he_1+ y \vert^n }  \bigg ) &= \frac{2ny_1 }{\vert y \vert^{n+2}}.
\end{align*} Moreover, by a similar argument to~\eqref{LxZU6},\begin{align*}
\bigg \vert  \frac 1 {(\vert  y \vert^2 -r^2)^s }  \bigg ( \frac 1 {\vert x - y \vert^n }-\frac 1 {\vert x_\ast - y \vert^n }  \bigg )\bigg \vert  &\leqslant   \frac{2n  h \vert  y_1 \vert }{(\vert  y \vert^2 -r^2)^s\vert he_1 - y \vert^{n+2}} \in L^1(\R^n \setminus B_r).
\end{align*} Thus, we obtain the result by taking the limit \(h\to 0\) in \eqref{CNZLU} and applying the Dominated Convergence Theorem to justify swapping the limit and the integral on the right-hand side. 
\end{proof}

{F}rom Proposition~\ref{cqGgE}, we obtain the following corollary. 

\begin{cor} \thlabel{OUqJk}
Let \(u\in C^{2s+\alpha}(B_1) \cap \mathscr L_s(\R^n)\) with \(\alpha>0\) and \(2s+\alpha\) not an integer. Suppose that \(u\) is antisymmetric and \((-\Delta)^s u = 0 \) in \(B_1\).

Then there exists a radially symmetric function \(\psi_s \in C(\R^n) \) satisfying \begin{align}
\frac{C^{-1}}{1+\vert y \vert^{n+2s+2}} \leqslant \psi_s (y) \leqslant \frac{C}{1+\vert y \vert^{n+2s+2}} \qquad \text{for all } y \in \R^n, \label{nhvr2}
\end{align}
for some~$C>1$, such that \begin{align*}
\frac{\partial u}{\partial x_1} (0) &=  \int_{\R^n} y_1\psi_s (y) u(y) \dd y.
\end{align*} In particular, if \(u\) is non-negative in \(\R^n_+\) then \[C^{-1} \Anorm{u} \leqslant \frac{\partial u}{\partial x_1} (0)  \leqslant C \Anorm{u}.\] 
\end{cor}

\begin{proof}
Let \begin{align*}
\psi_s (y) &:=n (n+2)\gamma_{n,s} \int_0^{\min \{ 1/\vert y \vert,1\}} \frac{r^{2s+n+1}}{(1 - r^2)^s}\dd  r. 
\end{align*} It is clear that \(\psi_s\in C(\R^n)\) and that there exists \(C>1\) such that~\eqref{nhvr2} holds. Since \(u\in \mathscr L_s(\R^n)\), we have that~\( \Anorm{u} <+\infty\), so it follows that \begin{align}
\bigg \vert \int_{\R^n}y_1 \psi_s (y) u(y) \dd y \bigg \vert \leqslant C \Anorm{u} < +\infty. \label{9B6Vm}
\end{align} If we multiply \(\partial_1 u (0)\) by \(r^{n+1}\) then integrate from 0 to 1, Proposition~\ref{cqGgE} and~\eqref{ZrGnZFNg} give that
\begin{align}
\frac 1 {n+2} \frac{\partial u}{\partial x_1} (0) = \int_0^1 r^{n+1} \frac{\partial u}{\partial x_1} (0) \dd r = n \gamma_{n,s}\int_0^1 \int_{\R^n \setminus B_r} \frac{r^{2s+n+1}y_1 u(y)}{(\vert y \vert^2 - r^2)^s\vert y \vert^{n+2}} \dd y \dd r .\label{YYJ5o}
\end{align} At this point, let us observe that if we formally swap the integrals in~\eqref{YYJ5o} and then make the change of variables \(r=\vert y \vert \tilde r\), we obtain \begin{align}
 &n \gamma_{n,s}  \int_{\R^n \setminus B_1} \int_0^1 \frac{r^{2s+n+1}y_1 u(y)}{(\vert y \vert^2 - r^2)^s\vert y \vert^{n+2}}\dd r \dd y +  n \gamma_{n,s}\int_{B_1} \int_0^{\vert  y \vert}  \frac{r^{2s+n+1}y_1 u(y)}{(\vert y \vert^2 - r^2)^s\vert y \vert^{n+2}}\dd r \dd y \nonumber\\
&= n \gamma_{n,s}  \int_{\R^n \setminus B_1} \bigg ( \int_0^{1/\vert y \vert} \frac{\tilde r^{2s+n+1}}{(1 -\tilde r^2)^s}\dd \tilde r \bigg ) y_1 u(y)\dd y +  n \gamma_{n,s}\int_{B_1}\bigg ( \int_0^1   \frac{\tilde r^{2s+n+1}}{(1 - \tilde r^2)^s}\dd \tilde r\bigg )y_1 u(y) \dd y \nonumber \\
&= \frac 1 {n+2} \int_{\R^n}y_1 \psi_s(y) u(y) \dd y .  \label{LeVoi}
\end{align} By~\eqref{9B6Vm}, equation~\eqref{LeVoi} is finite. Hence, Fubini's theorem justifies changing the order of integration in \eqref{YYJ5o} and that the right-hand side of~\eqref{YYJ5o} is equal to~\eqref{LeVoi} which proves the result. 
\end{proof}

At this point, we can give the proof of Theorem~\ref{oCuv7Zs2}. 

\begin{proof}[Proof of Theorem~\ref{oCuv7Zs2}]
We will begin by proving that \begin{align}
u(x)  \geqslant C x_1 \Anorm{u}  \qquad \text{for all } x\in B_{1/2}^+. \label{p6oBI}
\end{align}
To this end, we observe that since \(\vert x_\ast - y \vert^{n+2} \leqslant C \vert y \vert^{n+2} \) for all \(x \in B_1\) and~\(y\in \R^n\setminus B_1\),~\eqref{buKHzlE6} gives that \begin{align*}
\frac 1 {\vert x - y \vert^n }-\frac 1 {\vert x_\ast - y \vert^n } \geqslant C \frac{  x_1 y_1}{\vert y \vert^{n+2}} \qquad \text{for all } x\in B_1^+ {\mbox{ and }} y\in \R^n_+ \setminus B_1^+.
\end{align*}Hence, by Lemma~\ref{Gku6y}, for all~$x\in B_{1/2}^+$,  \begin{align*}
u(x) &= C\int_{\R^n \setminus B_1} \bigg ( \frac{1- \vert x \vert^2 }{\vert  y \vert^2 -1 } \bigg )^s \bigg ( \frac 1 {\vert x - y \vert^n }-\frac 1 {\vert x_\ast - y \vert^n }  \bigg ) u(y) \dd y  \\
&\geqslant Cx_1 \int_{\R^n \setminus B_1} \frac{ y_1 u(y)}{\big ( \vert  y \vert^2 -1 \big )^s\vert y \vert^{n+2}}    \dd y 
\end{align*} where we used that \((-y_1)u(y_\ast)=y_1u(y)\) and that \(u\geqslant 0\) in \(\R^n_+\). Then
Proposition~\ref{cqGgE} with \(r=1\) gives that \begin{align*}
u(x) &\geqslant Cx_1 \frac{\partial u}{\partial x_1}(0)  \qquad \text{for all } x\in B_{1/2}^+.
\end{align*} Finally, Corollary~\ref{OUqJk} gives~\eqref{p6oBI}. 

Next, we will prove that  \begin{align}
u(x)  \leqslant C x_1 \Anorm{u}  \qquad \text{for all } x\in B_{1/2}^+. \label{3abqf}
\end{align} Similar to above, for all \(x\in B_{1/2}^+\) and~\(y\in  \R^n_+ \setminus B_1^+\), we have that~\(\vert x- y \vert \geqslant \frac 1 2 \vert y \vert\), so~\eqref{LxZU6} gives that\begin{align*}
 \frac 1 {\vert x - y \vert^n }-\frac 1 {\vert x_\ast - y \vert^n }  \leqslant \frac{C  x_1 y_1}{\vert y \vert^{n+2}} \qquad \text{for all } x\in B_{1/2}^+ {\mbox{ and }} y\in  \R^n_+ \setminus B_1^+.
\end{align*} As before, using Lemma~\ref{Gku6y}, we have that, for all~$x\in B_{1/2}^+$,\begin{align*}
u(x) &= C\int_{\R^n \setminus B_1} \bigg ( \frac{1- \vert x \vert^2 }{\vert  y \vert^2 -1 } \bigg )^s \bigg ( \frac 1 {\vert x - y \vert^n }-\frac 1 {\vert x_\ast - y \vert^n }  \bigg ) u(y) \dd y  \\
&\leqslant C x_1 \int_{\R^n \setminus B_1} \frac{ y_1 u(y)}{\big ( \vert  y \vert^2 -1 \big )^s\vert y \vert^{n+2}}    \dd y.
\end{align*}Then Proposition~\ref{cqGgE} and Corollary~\ref{OUqJk} give that \begin{align*}
u(x) &\leqslant Cx_1 \frac{\partial u}{\partial x_1}(0) \leqslant C x_1 \Anorm{u}  \qquad \text{for all } x\in B_{1/2}^+,
\end{align*} which is~\eqref{3abqf}.

{F}rom~\eqref{p6oBI} and~\eqref{3abqf} the result follows easily. 
\end{proof}

\section{A proof of~\eqref{LA:PAKSM} when~$c:=0$
that relies on extension methods}\label{APPEEXT:1}

We consider the extended variables~$X:=(x,y)\in\R^n\times\R$.
Then, a solution~$u$ of~$(-\Delta)^su=0$ in~$\Omega^+$ can be seen as the trace along~$\Omega^+\times\{0\}$ of its $a$-harmonic extension~$U=U(x,y)$ satisfying
$$ {\rm div}_{\!X}\big(|y|^a\nabla_{\!X} U\big)=0\quad{\mbox{ in }}\,\R^{n+1},$$
where~$a:=1-2s$, see Lemma~4.1 in~\cite{MR2354493}.

We observe that the function~$V(x,y):=x_1$ is also a solution of the above equation.
Also, if~$u$ is antisymmetric, then so is~$U$, and consequently~$U=V=0$ on~$\{x_1=0\}$.

As a result, by the boundary Harnack inequality (see~\cite{MR730093}),
\begin{equation}\label{LA:PAKSM:2}
\sup_{\tilde \Omega^+\times(0,1)} \frac{U}{V}  \leqslant C \inf_{ \tilde \Omega^+\times(0,1)} \frac{U}{V} .
\end{equation}

In addition,
$$ \sup_{\tilde \Omega^+\times(0,1)} \frac{U}{V}\ge
\sup_{\tilde \Omega^+\times\{0\}} \frac{U}{V}
=\sup_{\tilde \Omega^+} \frac{u(x)}{x_1}$$
and similarly
$$ \inf_{\tilde \Omega^+\times(0,1)} \frac{U}{V}\le\inf_{\tilde \Omega^+} \frac{u(x)}{x_1}.$$
{F}rom these observations and~\eqref{LA:PAKSM:2} we obtain~\eqref{LA:PAKSM}
in this case.

\section*{Acknowledgements}

All the authors are members of AustMS. 
SD is supported by
the Australian Research Council DECRA DE180100957
``PDEs, free boundaries and applications''.
JT is supported by an Australian Government Research Training Program Scholarship. 
EV is supported by the Australian Laureate Fellowship
FL190100081
``Minimal surfaces, free boundaries and partial differential equations''.

JT would also like to thank David Perrella for his interesting and fruitful conversations. 

\printbibliography

 \vfill
\end{document}